\documentclass[11pt]{article}

\usepackage[utf8]{inputenc}
\usepackage{amsmath,amsfonts,stmaryrd,amsthm,thmtools,hyperref,graphicx}
\usepackage{caption,subcaption}

\usepackage{fullpage}
\usepackage{authblk}

\newtheorem{theorem}{Theorem}
\newtheorem{lemma}[theorem]{Lemma}
\newtheorem{proposition}[theorem]{Proposition}
\newtheorem{corollary}[theorem]{Corollary}
\theoremstyle{definition}

\newtheorem{assumption}[theorem]{Assumption}
\newtheorem{remark}[theorem]{Remark}

\numberwithin{definition}{section}
\numberwithin{proc}{section}
\numberwithin{equation}{section}
\numberwithin{condition}{section}
\numberwithin{assumption}{section}
\numberwithin{proposition}{section}
\numberwithin{theorem}{section}
\numberwithin{lemma}{section}
\numberwithin{remark}{section}
\numberwithin{claim}{section}
\numberwithin{observation}{section}
\numberwithin{corollary}{section}

\usepackage{enumerate}

\usepackage{xcolor}
\definecolor{webgreen}{rgb}{0,.5,0}
\definecolor{Maroon}{HTML}{800000}
\usepackage{hyperref}
\hypersetup{colorlinks, breaklinks, urlcolor=Maroon, linkcolor=Maroon, citecolor=webgreen} 

\usepackage[textwidth=30mm]{todonotes}

\usepackage[numbers]{natbib}
\newcommand*{\doi}[1]{doi:\,\texttt{\href{http://dx.doi.org/#1}{#1}}}

\newcommand{\bE}{\mathbb{E}}
\newcommand{\bG}{\mathbb{G}}
\newcommand{\bN}{\mathbb{N}}
\newcommand{\bZ}{\mathbb{Z}}
\newcommand{\bR}{\mathbb{R}}

\newcommand\bfd{{\mathbf d}}

\newcommand\dsG{{\mathbb G}}

\usepackage{physics}
\usepackage{mathtools}
\allowdisplaybreaks

\newcommand{\ind}[1]{\mathbf{1}_{#1}}

\newcommand{\R}{{\mathbb{R}}}

\newcommand{\N}{{\mathbb{N}}}

\newcommand{\G}{{\mathbb{G}}}

\newcommand{\eps}{\epsilon}

\newcommand\cA{{\cal{A}}}
\newcommand\cB{{\cal{B}}}
\newcommand\cC{{\cal{C}}}
\newcommand\cD{{\cal{D}}}

\newcommand\cF{{\cal{F}}}

\newcommand\cL{{\cal{L}}}

\newcommand\cN{{\cal{N}}}

\newcommand{\E}[1]{{\mathbb{E}}\left[#1\right]}

\newcommand{\Vv}[1]{{\mathrm{Var}}(#1)}

\newcommand{\bP}{{\mathbb{P}}}
\newcommand{\p}[1]{{\mathbb{P}}(#1)}
\newcommand{\pl}[1]{{\mathbb{P}}\left(#1\right)}

\newcommand{\hcf}{\text{hcf}}

\usepackage[authormarkup=none]{changes}

\definechangesauthor[name={Mat},color=red]{M}

\definechangesauthor[name={Guillem},color=blue]{G}

\usepackage[sort&compress,capitalize,nameinlink]{cleveref}
\crefname{app}{Appendix}{Appendices}
\crefname{assumption}{Assumption}{Assumptions}

\crefrangeformat{equation}{\upshape(#3#1#4)\textendash(#5#2#6)}

\newcommand\dsCM{{\mathbb{CM} }}
\newcommand\Gn{{\dsG}_{n}}
\newcommand\CMn{{\dsCM}_{n}}
\newcommand\hD{\hat{D}}

\title{Largest component of subcritical random graphs with given degree sequence}

\author[\oplus]{Matthew Coulson}
\author[\sharp]{Guillem Perarnau}
\affil[\oplus]{\small\it Departament de Matem\`atiques (MAT), Universitat Polit\`ecnica de Catalunya (UPC), Barcelona, Spain. \\
Email:~{\tt matthew.john.coulson@upc.edu@upc.edu}.}
\affil[\sharp]{\small\it Departament de Matem\`atiques (MAT) and IMTECH, Universitat Polit\`ecnica de Catalunya (UPC), Barcelona, Spain.\\
\small\it Centre de Recerca Matem\`atica,  Barcelona, Spain\\ Email:~{\tt guillem.perarnau@upc.edu}.}

\begin{document}
\maketitle

\begin{abstract}
   We study the size of the largest component of two models of random graphs with prescribed degree sequence, the configuration model (CM) and the uniform model (UM), in the (barely) subcritical regime. For the CM, we give upper bounds that are asymptotically tight for certain degree sequences. These bounds hold under mild conditions on the sequence and improve previous results of Hatami and Molloy on the barely subcritical regime. For the UM, we give weaker upper bounds that are tight up to logarithmic terms but require no assumptions on the degree sequence. In particular, the latter result applies to degree sequences with infinite variance in the subcritical regime.
\end{abstract}

\section{Introduction}\label{sec:introduction}

 Let $[n]\coloneqq \{1,\dots,n\}$ be a set of $n$ vertices.  Let $\bfd_n=(d_1,\dots,d_n)$ be a degree sequence with $m\coloneqq \sum_{i\in [n]} d_i$ an even positive integer. Without loss of generality, we will assume that $d_1\leq \dots\leq d_n$. Additionally, we may assume that $d_1\geq 1$; if there are elements with degree $0$ we can remove them and study the remainder sequence. Let $\Delta=\Delta_n$ be the maximum degree of $\bfd_n$.

The \emph{configuration model}, denoted by $\CMn=\CMn(\bfd_n)$, is
the random multigraph on $[n]$ generated by giving $d_{i}$ \emph{half-edges} (or stubs) to vertex $i$, and then pairing the half-edges uniformly at random. 
The \emph{uniform model}, denoted by $\G_n=\G_n(\bfd_n)$, is the random simple graph on $[n]$ obtained by choosing a simple graph uniformly at random among all graphs on $[n]$ where vertex $i$ has degree $d_i$. Throughout the paper, all the results on the uniform model will assume that the sequence $\bfd_n$ is \emph{graphical}; that is, there exists at least one graph on $[n]$ with such degree sequence.

For any graph $G$ on $[n]$, let $L_1(G)$ denote the order of a largest component. A central problem in random graph theory is to find a parameter of the model $\alpha$ such that $L_1$ undergoes a phase transition at $\alpha=\alpha_0$. The set of parameters is then divided into subcritical ($\alpha<\alpha_0$), critical ($\alpha\sim\alpha_0$) and supercritical ($\alpha>\alpha_0$). A further problem is the study the critical window: that is, to find parameters $\beta^-,\beta^+$, such that $L_1$ behaves essentially the same for any $\alpha\in (\alpha_0-\beta^-, \alpha_0+\beta^+)$ and the critical region is further divided into barely subcritical, $\alpha_0 - \alpha >\beta^-$  and barely supercritical, where $\alpha - \alpha_0 >\beta^+$.

The main goal of this paper is to study the largest component phase transition $L_1(\CMn)$ and $L_1(\Gn)$ in the subcritical and the barely subcritical regimes.
Generally speaking, the study of configuration model is simpler due to the existence of an explicit model with good independence properties. In contrast, most of the results existing for $\Gn$ arise from $\CMn$ by observing that the probability that $\CMn$ generates a simple graph is sufficiently large.

In order to understand the phase transition, define
\begin{align}
Q=Q_n(\bfd_n) \coloneqq \frac{1}{m}\sum_{i\in [n]} d_i(d_i-2) \label{def:Q}\;,\\
R=R_n(\bfd_n) \coloneqq \frac{1}{m}\sum_{i\in [n]} d_i(d_i-2)^2\label{def:R}\;.
\end{align}
In the first part of the paper, we will focus on the case $Q_n\leq 0$. It is easy to check that the bound on $Q_n$ implies $\Delta_n=O(\sqrt{n})$ and $m\leq 2n$. Also note the implicit bound on the maximum degree $\Delta_n=O(n^{1/3}R_n^{1/3})$ obtained by just considering the contribution of a vertex of maximum degree to $R_n$.

Let $D_n$ be the degree of a uniform random vertex and let $\hD_n$ be its size-biased distribution; that is, for $k\geq 1$,
\begin{align}
\p{\hD_n=k} = \frac{k\p{D_n=k}}{m} \label{def:size}\;.
\end{align}
For $b,h\in \mathbb{N}$, let $\cL(b, h):= \{b+h k: k\in \mathbb{Z}\}$ be the integer lattice containing $b$ with step $h$. Let $h_n$ be the largest integer $h$ such that $\bP(D_n \in \cL(b,h))=1$ for some $b\in \bN$.

We will study the $\CMn$ under the following conditions on the degree sequence:
\begin{assumption}\label{assum:CM} There exists a discrete random variable $D$ supported on $\mathbb{Z}_{\geq 0}$ such that
    \begin{enumerate}[(i)]	
        \item\label{CM1}  $D_n\to D$ in distribution;
        \item\label{CM2}  $Q_n \to 0$;
        \item\label{CM3} $\bP(D \not\in\{0,2\})>0$;
        \item\label{CM4} if $h$ is the largest integer such that $\bP(D \in \cL(b,h))$ for some $b\in \bN$, then $h_n=h$ for all $n$.
        \item\label{CM5} $\bE(D_n^4) \leq \Delta_n^{1/2}$
        \end{enumerate}
        
\end{assumption}
\begin{remark}
Conditions (i)-(iii) are usual in this setting. In particular, they imply that $R_n$ is bounded away from zero, which will be often used in the proofs.

 Condition (iv) simply asks that the limiting degree distribution $D$ has the same step as the random variables $D_n$ which converge to it. This restriction is not particularly strong and forbids no limiting degree sequence, only the way in which we converge to it.
 
Condition (v) is the most restrictive one. As $Q_n=o(1)$, we have $\bE[D^2_n]=O(1)$, which implies $\bE[D^4_n]=O(\Delta_n^2)$. Thus, this condition can be understood as a ``polynomial limitation'' on the contribution of large degree vertices to the fourth moment. It would be interesting to see up to which point a condition on the fourth moment is needed.
\end{remark}


Our first result upper bounds the size of the largest component when $Q$ is not too large with respect to $R$.
\begin{theorem}\label{cor:CM2}
Let $\epsilon>0$. Let $\bfd_n$ be a degree sequence satisfying \cref{assum:CM} and $\Delta |Q|= o(R)$. If $Q\leq -\omega(n) n^{-1/3}R^{2/3}$ for some $\omega(n)\to \infty$, then
\begin{align}\label{eq:second_bound}
	\pl{L_1(\CMn(\bfd_n))\leq (1+\epsilon)\frac{2R}{Q^2}\log \left(\frac{|Q|^3 n}{R^2}\right)} =1-o(1)\;.
\end{align}
\end{theorem}

\begin{remark}
As noted in~\cite{hatami2012scaling}, under the condition $|Q|\Delta=o(R)$ the critical window is $|Q|=O(n^{-1/3}R^{2/3})$. Therefore, \cref{cor:CM2} bounds the largest component in the whole barely subcritical regime. See~\cref{sec:prev} for further discussion.
\end{remark}

%
%

\noindent Let $\eta\coloneqq \eta_n =\hD_n-2$ and consider its moment generating function
\begin{align}
\varphi(\theta) \coloneqq \varphi_n(\theta) = \E{e^{\theta \eta_n}} \label{def:MGF}\;.
\end{align}
\cref{cor:CM2} is in fact a consequence of a more general result that does not require a bound of $Q$ in terms of $R$. 
\begin{theorem}\label{thm:CM1}
Let $\epsilon>0$. Let $\bfd_n$ be a degree sequence satisfying \cref{assum:CM} and $\Delta_n \leq n^{1/6}$. Let $\theta_0\in (0,1)$ be the smallest solution $\theta$ of $\varphi'(\theta)=0$. Define
\begin{align}\label{def:T}
T_n\coloneqq \frac{1}{\log (\varphi(\theta_0))^{-1}}\cdot \log\left(\frac{(\log(\varphi(\theta_0))^{-1})^{3/2}}{\varphi''(\theta_0)^{1/2}} \E{D_n e^{\theta_0 D_n}}n\right)\;.
\end{align}
If $\theta_0 m\geq \omega(n)T_n$ for some $\omega(n)\to \infty$, then
\begin{align}\label{eq:first_bound}
	\p{L_1(\CMn(\bfd_n))\leq (1+\epsilon)T_n} =1-o(1)\;.
\end{align}
\end{theorem}



\begin{remark}
 The value $\theta_0$ exists and is bounded as $n\to \infty$. We have $\varphi(0)=1$ and $\varphi'(0)=Q<0$.
Recall that $\eta$ is supported in $\{-1,0,1,\dots\}$.
By~\cref{assum:CM}, $\{D \geq 3\}$ happens with positive probability and so if we define $p:= \bP(D\geq 3)$, then $\bP(\eta\geq 1) \geq p$ and $\varphi(\theta) \geq (1-p)e^{-\theta} + pe^\theta$.
So $\varphi(\theta) \to \infty$ as $\theta \to \infty$ and it must at some point have positive derivative. Thus, there exists $\theta_0\in (0,1)$ such that $\varphi'(\theta_0)=0$.
\end{remark}


%

It is interesting to understand if these results also hold in the uniform setting. One can use the following result to transfer from $\CMn$ to $\Gn$.
\begin{theorem}[Janson~\cite{janson2009probability}]\label{thm:simple}
Let $\bfd_n$ be a degree sequence satisfying $m=\Theta(n)$ and $\E{D^2_n}=O(1)$. Then 
\begin{align}\label{eq:simple}
\p{\CMn(\bfd_n) \text{ simple}}= \exp\left(-\frac{1}{m}\sum_{i\in [n]} d_i^2\right) > 0\;,
\end{align}
and conditioned on being simple, $\CMn$ has the same law as $\Gn$. Therefore, any result that holds with probability $1-o(1)$ for $\CMn(\bfd_n)$, also holds with probability $1-o(1)$ for $\Gn(\bfd_n)$.
\end{theorem}
As \cref{cor:CM2} and \cref{thm:CM1} assume that $Q_n\leq 0$, we have that $\E{D^2_n}=O(1)$ and we can use \cref{thm:simple} to transfer their conclusions to $\Gn$, provided that their hypothesis are satisfied.

The second part of our paper focuses on the size of the largest component in the (barely) subcritical regime of $\Gn$ without further assumptions on the degree sequence. The lack of a tractable model for $\Gn$ hampers its analysis and the upper bounds obtained are weaker than the ones obtained for $\CMn$ and probably not of the right order.

Let $S_*$ be a smallest set of vertices of largest degree that satisfies
\begin{align}\label{eq:neg_set}
\sum_{u\in [n]\setminus S_*} d_u(d_u-2) \leq 0
\end{align}
and define
\begin{align}\label{eq:surplus}
m_*=\sum_{v\in S_*} d_v \;.
\end{align}
In particular, if $Q\leq 0$, then $S_*=\emptyset$ and $m_*=0$.

For any $m_0\geq 0$ and $Q_0\leq 0$, we call $\bfd_n$ an \emph{$(m_0,Q_0)$-subcritical} degree sequence if there exists $S\subseteq [n]$ with $\sum_{v\in S}d_v\leq m_0$ and 
\begin{align*}
\frac{1}{m}\sum_{w\in [n]\setminus S} d_w(d_w-2) \leq Q_0\;.
\end{align*}

Our most general result on $\Gn$ is the following.
\begin{theorem}\label{thm:UM1}
Let $\bfd_n$ be an $(m_0,Q_0)$-subcritical degree sequences for some parameters satisfying $m_0\geq 3m_*$, $m_0|Q_0|\geq (\Delta |Q_0|+R)\log\left(\frac{nQ_0^2}{\Delta|Q_0|+R}\right)$ and $Q_0^2 n \geq \omega(n) m_0$ for some $\omega(n)\to \infty$. Then,
\begin{align*}
	\pl{L_1(\Gn(\bfd_n))=O(m_0/|Q_0|)}=1-o(1)\;.
\end{align*}
\end{theorem}

\begin{remark}[Infinite degree variance]
The main strength of \cref{thm:UM1} is that it applies to degree sequences with subcritical behaviour but infinite degree variance. To our knowledge, the only results available in this setting are of the form $L_1(\Gn(\bfd_n))=o(n)$~\cite{bollobas2015old,jprr2018}. Note that, even if such results were available for $\CMn$,~\cref{thm:simple} is not strong enough to transfer them to $\Gn$.
\end{remark}

\begin{remark}[The $Q\leq 0$ case]\label{rem:Q<0}
To compare it with previous work, let us get more explicit results for the case $Q\leq 0$ (i.e. $\bE[D^2_n]\leq 2\bE[D_n]$). In this case, $m_*=0$ and we can choose $Q_0\coloneqq Q$ and $m_0\coloneqq (\Delta+R/|Q|)\log\left(\frac{nQ_0^2}{\Delta|Q_0|+R}\right)$. Also note that $Q\leq 0$ implies $R=O(\Delta)$.
\begin{itemize}

\item[(a)] If $|Q|$ is bounded away from zero, then all conditions in \cref{thm:UM1} are satisfied and
\begin{align}\label{eq:UB_a}
L_1(\Gn(\bfd_n))=O(\Delta \log n) \;.
\end{align}
%
%

\item[(b)] If $|Q|=o(1)$, then we split depending on how $R$ and $\Delta|Q|$ compare to each other.
\begin{itemize}
\item[(b.1)] If $\Delta|Q|=O(R)$, then for any $Q\leq -\omega(n) n^{-1/3} R^{1/3}$,
\begin{align}\label{eq:UB_b1}
L_1(\Gn(\bfd_n))=O\left(\frac{R}{Q^2}\log\left(\frac{nQ^2}{R}\right)\right)\;,
\end{align}
obtaining a weaker version of \cref{cor:CM2}, that is valid for all degree sequences.

\item[(b.2)] If $R=O(\Delta |Q|)$, then for any $Q\leq -\omega(n) n^{-1/2} \Delta^{1/2}$, 
\begin{align}\label{eq:UB_b2}
L_1(\Gn(\bfd_n))=O\left(\frac{\Delta}{|Q|}\log\left(\frac{nQ}{\Delta}\right)\right)\;.
\end{align}
\end{itemize}
\end{itemize}

\end{remark}

\begin{remark}
If $Q_0^2 n =O(m_0)$, then the behaviour of $\Gn$ is no longer (barely) subcritical. It is interesting to study the size of the largest component in this case.
\end{remark}

We finally provide the existence of infinitely many degree sequences that show the tightness of some of our upper bounds.
\begin{proposition}\label{prop:UM2}
For any $Q<0$, $\Delta= o(\sqrt{n})$ and $\log n=o(\Delta)$, there exists a degree sequence $\widehat\bfd_n$ with $\Delta_n(\widehat\bfd_n)=\Delta$, $Q_n(\widehat\bfd_n)\sim Q$ and $R=R_n(\widehat\bfd_n)\sim \Delta$, such that
\begin{align*}
	\pl{L_1(\Gn(\widehat\bfd_n)) \geq (1+o(1))\frac{2R}{Q^2}\log\left(\frac{n}{R^2}\right)}=1-o(1)\;.
\end{align*}
\end{proposition}

\begin{remark}\label{rem:LB}
We can compare the lower bound in~\cref{prop:UM2} with our upper bounds.
The degree sequence $\widehat\bfd_n$ satisfies $R\sim \Delta$, so $\Delta |Q|=O(R)$. In the case $\Delta<n^{1/2-\delta}$ for some constant $\delta>0$, the proposition gives a family of degree sequences for which~\cref{eq:UB_a} is of the right order.

While \cref{prop:UM2} is only stated for $Q$ bounded away from zero, one could similarly define degree sequences $\widehat \bfd_n$ for which $Q=o(1)$, in which case $\Delta |Q|=o(R)$. Provided that $Q\leq -\omega(n)n^{-1/3}R^{2/3}$ for some $\omega(n)\to \infty$, one obtains the lower bound in \cref{eq:final} that coincides asymptotically with \cref{cor:CM2} and, up to logarithmic terms, with~\cref{eq:UB_b1}. (See \cref{rem:Qto0}.)
\end{remark}

\subsection{Previous work}\label{sec:prev}

The foundational paper of Erd\H os and R\'enyi~\cite{erdHos1960evolution} located the phase transition for the existence of a linear order component in a uniformly chosen graph on $n$ vertices and $m$ edges, $\dsG(n,m)$, showing that the order of the largest component undergoes a double jump at $m=n/2$, in particular $L_1(\dsG(n,m))=O(\log{n})$ if $m\leq  cn$ and $c<1/2$, $L_1(\dsG(n,m))=\Theta(n^{2/3})$ if $m=  n/2$,  and $L_1(\dsG(n,m))=\Theta(n)$ if $m\leq  cn$ and $c>1/2$. This result can be easily transferred to the Binomial random graph $\dsG(n,p)$ with $p=2m/n$, which has become the reference model for random graphs. The size of the largest component in all regimes is well understood, see e.g. Sections 4 and 5 in~\cite{vanderhofstad2016}.

The study of the phase transitions in random graphs with given degree sequences was pioneered by Molloy and Reed~\cite{molloy1995critical}. The so-called \emph{Molloy-Reed criterion} determines the phase transition at $Q=0$, provided that the degree sequence satisfies a number of technical conditions. The criterion has been extended to degree sequences with bounded degree variance~\cite{janson2009new} and uniformly integrable sequences~\cite{bollobas2015old}, providing the asymptotic value of $L_1$ in the supercritical regime $Q>0$ in terms of the survival probability of an associated branching process, similarly as in the $\dsG(n,p)$ case. Interestingly, the criterion is no longer valid for general degree sequences due to the presence of high degree vertices (hubs) or an extremely large number of degree $2$ vertices. Joos et al.~\cite{jprr2018} gave an extended criterion that determines whether any given degree sequence typically produces a linear order component.

While the behaviour of the largest component in the supercritical regime resembles the simpler Erd\H os-R\'enyi model, this does not happen in the subcritical one, when $Q<0$. Trivially, we have $L_1(\Gn)\geq \Delta+1$ which could be much larger than logarithmic. In~\cite{molloy1995critical} the authors showed that $L_1(\CMn)=O(\Delta^2 \log{n})$ for subcritical sequences. More precise results are known for power-law degree sequences. Durrett~\cite{durrett2007random} conjectured\footnote{In fact, this was conjectured for a slightly different model where the degrees are i.i.d. copies of $D_n$ conditioned on their sum being even.} that if $\p{D_n=k}\sim c k^{-\gamma}$ for some $\gamma>3$ and $c>0$, then $L_1(\CMn)=O(\Delta)$. In this setting, $\gamma>3$ implies $\E{D_n^{\gamma-1}}=O(1)$. Pittel~\cite{pittel2008largest} showed that $L_1(\CMn)=O(\Delta\log n)$ for subpower-law distributions. Janson~\cite{janson2008largest} proved a strong version of the conjecture: if $\p{D_n\geq k}= O(k^{1-\gamma})$ for some $\gamma>3$, then
\begin{align}\label{eq:Janson_bound}
L_1(\CMn) = \frac{\Delta}{|Q|}+o(n^{1/(\gamma-1)})
\end{align}
For power-law distributions, we have $\Delta=\Theta(n^{1/(\gamma-1)})$ with high probability, and the second term is negligible. From the intuitive point of view, the largest component is obtained by starting a subcritical branching process with expected offspring $1+Q$ from each vertex adjacent to the vertex of largest degree. The expected total progeny of such process is $1/|Q|$. 
One can interpret the result of ~\cref{thm:UM1} in a similar spirit: in the largest component there might be at most $O(m_0)$ edges and from each of these edges a piece of size $O(1/|Q_0|)$ hangs. For $Q<0$,~\eqref{eq:Janson_bound} can be compared to the weaker bound \cref{eq:UB_a} that holds regardless of the shape of the degree sequence tail. As shown in~\cref{rem:Q<0}, for general degree sequences~\cref{eq:UB_a} cannot be  improved.


The critical regime has attracted a lot of interest in recent years~\cite{dhara2017critical,dhara2020heavy,hatami2012scaling,van2019component, joseph2010component,riordan2012phase} with several papers specialising on the finite second or finite third moment cases. Here we focus on the results known for the barely subcritical regime. 

Riordan~\cite{riordan2012phase} showed that if $\Delta=O(1)$ then~\cref{eq:second_bound} holds, even more, one has asymptotic equality and control on the second order term. \cref{thm:CM1} can be seen as a generalization of the upper bound in~\cite{riordan2012phase} to a wider class of sequences that allows $\Delta\to \infty$ as $n\to \infty$.

Hatami and Molloy~\cite{hatami2012scaling} studied the critical window under some mild conditions on the degree sequence. 
They showed that $|Q|=O(n^{-1/3}R^{2/3})$ is the critical window of $\CMn$. Regarding the barely subcritical regime, for $Q\leq  - \omega(n) n^{-1/3}R^{2/3}$ with $\omega(n) \to \infty$, they showed that
\begin{align}\label{eq:result_HM}
L_1(\CMn)&=O\left(\sqrt{\frac{n}{|Q|}}\right)\;.
\end{align}
One can check that~\cref{eq:result_HM} coincides in order with~\cref{eq:second_bound} at the boundary of the critical window $|Q|=\Theta(n^{-1/3}R^{2/3})$, while~\cref{eq:second_bound} improves~\cref{eq:result_HM} in the whole barely subcritical regime, provided that \cref{assum:CM} holds.

Under infinite variance, the probability of $\CMn$ being simple can be exponentially small in $n$. Thus, only results that hold with exponentially high probability can be transferred from $\CMn$ to $\Gn$, see e.g.~\cite{bollobas2015old}. Another approach is to study $\Gn$ directly using the \emph{switching method}~\cite{jprr2018}. In both cases, the best bound given in the subcritical regime is $L_1(\Gn)=o(n)$. 
\cref{thm:UM1} provides the first explicit general bound to $L_1$ at subcriticallity  for infinite variance degree sequences. As discussed in \cref{rem:LB}, this bound cannot be substantially improved without further assumptions. 

It thus remains as an open question to determine the exact size of the largest component in the (barely) subcritical regime. 
Hofstad, Janson and \L uczak conjectured that $L_1(\CMn)$ is concentrated in this regime~\cite{van2019component}. Supported by the result of Riordan for constant maximum degree, we conjecture that the upper bound in~\cref{eq:first_bound} is asymptotically tight for all degree sequences that satisfy some assumption. Note that certain condition on the degree sequence is needed, as for some particular subcritical degree sequences $L_1(\CMn)$ is non-concentrated (see \cref{rem:non_concentrated}).
%
%

%
%

\section{A local limit theorem}
A local limit theorem estimates the probability distribution of a suitably rescaled sum of independent random variables, by the density function of a Gaussian random variable. Local limit theorems are a useful tool to determine the component size in random graphs~\cite{nachmias2010critical,riordan2012phase}. For our application, we will need the step distribution to allow for the existence of very large degrees, as well as the fact that the degree sequence may be supported on an lattice with step different than $1$. This prevents us from using classical results such as Berry-Esseen Theorem (see~\cite[Theorem 3.4.9]{Durrett}). Our goal is to develop a precise local limit theorem which will allow us to deal with our step distributions. Our result is based on previous local limit theorems by Doney~\cite{Doney} and Mukhin~\cite{Mukhin1,Mukhin2} from which we derive more explicit error bounds. In particular the main result of this section is following,

\begin{theorem}\label{thm:lllt}
 Let $X_1, X_2, \ldots, X_n$ be independent and identically distributed random variables taking values on $\cL(v_0, h)$.
 Define $S_n = \sum_{i=1}^n X_i$.
 Suppose that $\mu = \bE(X_1)=0$, $\sigma^2 = \Vv{X_1}$ and $\gamma = \bE|X_1|^3$, and let $\varphi(t)$ be the characteristic function of $X_1$. Then, 
\begin{align}\label{eq:thm:llt}
 \sup_{ w \in\cL(nv_0, h)} \left| \bP(S_n = w) - \frac{h}{\sqrt{2\pi n \sigma^2}} \exp \left(-\frac{w^2}{2n\sigma^2}\right) \right| \leq \frac{32h\gamma}{\sigma^4 n} + \frac{h}{\pi} \int_\frac{\sigma^2}{4\gamma}^\frac{\pi}{h} |\varphi(t)|^n dt.
\end{align}
\end{theorem}
To prove this theorem, we will require the following Fourier inverse theorems which can be found in the book of Durett~\cite{Durrett}.
\begin{theorem}[Continuous Fourier Inverse Theorem]\label{thm:ctns_inv}
 Suppose that $X$ is a random variable with characteristic function $\varphi_X(t)$.
 Suppose further that $\varphi_X(t)$ is integrable, then $X$ is a continuous random variable with density function $f(y)$ defined by
 $$
 f(y) = \frac{1}{2\pi} \int_{\bR} e^{-ity}\varphi_X(t) dt.
 $$
\end{theorem}

\begin{theorem}[Discrete Fourier Inverse Theorem]\label{thm:disc_inv}
Let $X$ be a random variable with characteristic function $\varphi_X(t)$. Suppose that there exists $h>0$ such that $\bP(X\in h\bZ) = 1$. Then for any $a \in h\bZ$,
$$
\bP(X=a) = \frac{h}{2\pi} \int_{-\frac{\pi}{h}}^\frac{\pi}{h} e^{-ita}\varphi_X(t) dt.
$$
\end{theorem}
Finally, as we are interested in local limit theorems it will be useful to note that the characteristic function of the standard normal distribution is given by $N(t) = e^{-\frac{t^2}{2}}$.
\begin{proof}[Proof of Theorem~\ref{thm:lllt}]
Let $\varphi(t)$ be the characteristic function of $X_1$ and $\psi_n(t)$ the characteristic function of $S_n$.
By basic properties of characteristic functions, it is easy to see that $\psi_n(t) = \varphi(t)^n$.
 By Theorems~\ref{thm:ctns_inv} and~\ref{thm:disc_inv} we may deduce that
 \begin{equation*}
  \bP(S_n = w) - \frac{h}{\sqrt{2\pi n \sigma^2}} \exp \left(-\frac{w^2}{2n\sigma^2}\right) = \frac{h}{2\pi} \int_{-\frac{\pi}{h}}^\frac{\pi}{h} e^{-itw}(\psi_n(t)-N(t\sigma\sqrt{n}))dt - \frac{h}{\pi}\int_\frac{\pi}{h}^\infty e^{-itw}N(t\sigma\sqrt{n})dt.
 \end{equation*}
Therefore, by applying various forms of the triangle law we obtain the bound
\begin{equation}\label{eq:llt_pbound_1}
 \left| \bP(S_n = w) - \frac{h}{\sqrt{2\pi n \sigma^2}}
 \exp \left(-\frac{w^2}{2n\sigma^2}\right) \right| \leq
 \frac{h}{2\pi} \int_{-\frac{\pi}{h}}^\frac{\pi}{h} |\psi_n(t)-
 N(t\sigma\sqrt{n})|dt + \frac{h}{\pi} \int_\frac{\pi}{h}^\infty N(t\sigma\sqrt{n})dt.
\end{equation}
To bound the first integral in~\eqref{eq:llt_pbound_1} we split it into three parts.
For $\eps>0$ (which we shall pick later) we have
\begin{equation}\label{eq:llt_3integrals}
 \int_{-\frac{\pi}{h}}^\frac{\pi}{h} |\psi_n(t)-
 N(t\sigma\sqrt{n})|dt \leq \int_{-\eps}^\eps |\psi_n(t)-
 N(t\sigma\sqrt{n})|dt + 2\int_{\eps}^\frac{\pi}{h} N(t\sigma\sqrt{n})dt + 2\int_{\eps}^\frac{\pi}{h} |\varphi(t)|^n dt
\end{equation}
The idea of this bound being that both $\psi_n(t)$ and $N(t\sigma\sqrt{n})$ only contribute a non-trivial amount to the left hand side of~\eqref{eq:llt_3integrals} for $t$ very close to $0$.
In bounding the first integral of~\eqref{eq:llt_3integrals} we will use the following lemma~\cite[Page 109, Lemma 1]{Petrov}.
\begin{lemma}
 Let $X_1, \ldots,  X_n$ be independent random variables with $\bE(X_i)=0$, $\sigma^2_i = \Vv{X_i}$ and $\gamma_i = \bE|X_i|^3$.
 Define
 \begin{align*}
  B_n := \sum_{i=1}^n \sigma^2_i && L_n := \frac{\sum_{i=1}^n \gamma_i}{B_n^{3/2}} && T_n := \frac{\sum_{i=1}^n X_i}{B_n^{1/2}}
 \end{align*}
 Let $f_n(t)$ be the characteristic function of $T_n$. Then,
 \begin{equation}\label{eq:esseen_ineq}
  |f_n(t) - e^{-\frac{t^2}{2}}| \leq 16L_n|t|^3 e^{-\frac{t^2}{3}}\; \text{ for }\; |t| \leq \frac{1}{4L_n}.
 \end{equation}
\end{lemma}
Clearly this is applicable in our setting, however we need to rescale first which allows us to deduce
 $$|\psi_n(t)- N(t\sigma\sqrt{n})| \leq 16\gamma n |t|^3 e^{-\frac{t^2\sigma^2 n}{3}}\; \text{ for }\; |t| \leq \frac{\sigma^2}{4\gamma}.
 $$
So, for any $\eps\leq \sigma^2/(4\gamma)$ we have
\begin{align}
 \int_{-\eps}^\eps |\psi_n(t)- N(t\sigma\sqrt{n})|dt &\leq 16\gamma n \int_{-\eps}^\eps |t|^3 e^{-\frac{t^2\sigma^2 n}{3}} dt\nonumber \\
 & \leq 16\gamma n \int_{-\infty}^\infty |t|^3 e^{-\frac{t^2\sigma^2 n}{3}} dt\nonumber \\
 & =  \frac{16\gamma}{\sigma^4 n} \int_{-\infty}^\infty |t|^3 e^{-\frac{t^2}{3}} dt = \frac{144\gamma}{\sigma^4 n}. \label{eq:bound_I1}
\end{align}
The next step is to bound the second term of~\eqref{eq:llt_3integrals}.
Note that we can combine this with bounding the second term of~\eqref{eq:llt_pbound_1}, so we require to give an upper bound on
\begin{equation}\label{eq:I2_bound}
 \int_\eps^\infty N(t\sigma\sqrt{n}) dt = \frac{1}{\sigma\sqrt{n}} \int_{\eps \sigma \sqrt{n}}^\infty e^{-\frac{t^2}{2}} dt= \frac{\sqrt{2\pi}}{\sigma \sqrt{n}} \bP(\cN(0,1)>\eps \sigma \sqrt{n}).
\end{equation}
We use the Chernoff's bound for the standard normal distribution, $\bP(\cN(0,1)>x) \leq e^{-\frac{x^2}{2}}$ and the simple inequality $e^{-x} \leq x^{-1/2}$ for $x>0$, to obtain
\begin{equation}\label{eq:I2_bound_final}
 \int_\eps^\infty N(t\sigma\sqrt{n}) dt = \frac{\sqrt{2\pi}}{\sigma \sqrt{n}} \bP(\cN(0,1)>\eps \sigma \sqrt{n}) \leq \frac{2\sqrt{\pi}}{\eps \sigma^2 n}
\end{equation}
Choosing $\eps = \sigma^2/(4\gamma)$ and combining~\eqref{eq:llt_pbound_1},~\eqref{eq:llt_3integrals},~\eqref{eq:bound_I1} and \eqref{eq:I2_bound_final}, we find that for any $w \in \cL(nv_0,h)$,
\begin{align}
\left| \bP(S_n = w) - \frac{h}{\sqrt{2\pi n \sigma^2}} \exp \left(-\frac{w^2}{2n\sigma^2}\right) \right| &\leq 
\left( \frac{72}{\pi} + \frac{16}{\sqrt{\pi}}\right) \frac{h\gamma}{\sigma^4 n} + \frac{h}{\pi} \int_\frac{\sigma^2}{4\gamma}^\frac{\pi}{h} |\varphi(t)|^n dt \nonumber \\
& \leq \frac{32h\gamma}{\sigma^4 n} + \frac{h}{\pi} \int_\frac{\sigma^2}{4\gamma}^\frac{\pi}{h} |\varphi(t)|^n dt.\label{eq:SndiffNorm}
\end{align}
Concluding the proof of the theorem.
\end{proof}

For the remainder of this section we will focus on bounding the integral term in the RHS of \eqref{eq:thm:llt}.
To this end we introduce the parameter $H_D(X)$, which generalises a similar parameter introduced by Mukhin~\cite{Mukhin2, Mukhin1}.
For a real-valued random variable $X$, we define $X^* = X-X'$ to be the \emph{symmetrisation} of $X$, where $X'$ is an independent copy of $X$.
Furthermore, for $\alpha \in\bR$ define $\langle \alpha \rangle$ to be the distance from $\alpha$ to the nearest integer.
Then for a random variable $X$ and $d\in \mathbb{R}$ we define the following parameter
\begin{align*}
 H(X,d) &:= \bE\langle X^* d \rangle^2.
\end{align*}
The parameter $H(X,d)$ measures in a certain sense how close is $X^*$ to be a random variable supported on a lattice with step $1/|d|$.

The following lemma from~\cite{Mukhin2} will be useful.
\begin{lemma}\label{lem:bound_cf_H}
 If $\varphi(t)$ is the characteristic function of the random variable $X$ then
 \begin{equation}
  4H\left(X,\frac{t}{2\pi}\right) \leq 1-|\varphi(t)| \leq 2\pi^2 H\left(X,\frac{t}{2\pi}\right)
 \end{equation}
\end{lemma}
We provide a full proof of the statement, for the sake of completeness.
\begin{proof}
 We look at the characteristic function of $X^*$, $\varphi^*(t)$.
 Note that $X^*$ is by definition symmetric around the origin and hence so is $\varphi^*(t)$.
 Writing $\cD(X^*)$ for the domain of $X^*$ which is discrete, we have
 \begin{equation}
  \varphi^*(t) = \frac{\varphi^*(t) + \varphi^*(-t)}{2} = \sum_{x \in \cD(X^*)} \frac{e^{itx} + e^{-itx}}{2} \bP(X^* = x) = \sum_{x \in \cD(X^*)} \cos(tx) \bP(X^* = x). \label{eq:chf_cos}
 \end{equation}
As $\cos(x)$ is symmetric around $\pi$ and periodic with period $2\pi$, we have the identity
$$
\cos(x) = \cos\left(2\pi\left\langle\frac{x}{2\pi}\right\rangle \right).
$$
We can use this identity to rewrite~\eqref{eq:chf_cos} as
\begin{equation}
 \varphi^*(t) = \sum_{x \in \cD(X^*)} \cos\left(2\pi\left\langle\frac{tx}{2\pi}\right\rangle \right) \bP(X^* = x). \label{eq:chf_identity}
\end{equation}
Consider the following bounds on $\cos(x)$ valid for $x\in[0,\pi]$,
$$
1 - \frac{x^2}{2} \leq \cos(x) \leq 1 - \frac{2 x^2}{\pi^2}.
$$
We can use this in combination with~\eqref{eq:chf_identity} to deduce that
\begin{equation}
 1 - 2\pi^2 \bE\left\langle \frac{X^* t}{2\pi} \right\rangle^2 \leq \varphi^*(t) \leq 1 - 8 \bE\left\langle \frac{X^* t}{2\pi} \right\rangle^2
\end{equation}
Finally, by definition of $X^*$, note that $\varphi^*(t) = \varphi(t)\varphi(-t) = |\varphi(t)|^2$.
As $|\varphi(t)| \in [0,1]$ we may deduce that
\begin{equation}
 1 - 2\pi^2 H\left( X, \frac{t}{2\pi} \right) \leq |\varphi(t)| \leq 1 - 4 H\left( X, \frac{t}{2\pi} \right)
\end{equation}
which may easily be rearranged to give the statement of the lemma.
\end{proof}
For $D\in \mathbb{N}$ and a random variable $X$ we define
\begin{align}\label{eq:H_D}
H_D(X) := \inf_{\frac{1}{4D} \leq d \leq \frac{1}{2D}} H(X,d)
\end{align}
The following lemma of Mukhin~\cite{Mukhin1} bounds $H(X,d)$ in terms of this new parameter, 
\begin{lemma}\label{lem:bound_HfromHD}
 For any random variable $X$, $d \in \bR$ and
 $D\in \mathbb{N}$ with $2D |d|\leq 1$ we have
 $$
 H(X, d) \geq 4D^2 |d|^2 H_D(X).
 $$ 
\end{lemma}
We may now apply Lemma~\ref{lem:bound_cf_H} to  give an explicit upper bound on the integral term in~\eqref{eq:thm:llt} as follows. Recall that $X_1$ is a lattice random variable with step $h$. By Lemma~\ref{lem:bound_cf_H} and using $\ln(1/x)\geq 1-x$ for $x>0$,
\begin{equation}\label{eq:int_phi_n_bound}
 \int_\frac{\sigma^2}{4\gamma}^\frac{\pi}{h} |\varphi(t)|^n dt \leq \int_\frac{\sigma^2}{4\gamma}^\frac{\pi}{h}  e^{-n(1-|\varphi(t)|)} dt \leq \int_\frac{\sigma^2}{4\gamma}^\frac{\pi}{h}  e^{-4nH(X,\frac{t}{2\pi})} dt.
\end{equation}
Now, note that the upper limit of the integral in~\eqref{eq:int_phi_n_bound} is $\pi/h$.
So, as $(\pi/h)/(2\pi) = 1/(2h)$ we may apply Lemma~\ref{lem:bound_HfromHD} with $D=h$ and $d=t/2\pi \leq 1/2h$ to deduce that
\begin{align}\label{eq:int_phi_n_bound_2}
  \int_\frac{\sigma^2}{4\gamma}^\frac{\pi}{h} |\varphi(t)|^n dt & \leq \int_\frac{\sigma^2}{4\gamma}^\frac{\pi}{h}  e^{-\frac{4nh^2H_h(X)t^2}{\pi^2}} dt \leq \int_\frac{\sigma^2}{4\gamma}^\infty  e^{-\frac{4nh^2H_h(X)t^2}{\pi^2}}dt \nonumber \\ 
  & = \frac{\pi^{3/2}}{2h(nH_h(X))^{1/2}} \bP\left(\cN(0,1)>\frac{\sigma^2h(nH_h(X))^{1/2}}{\sqrt{2}\pi\gamma}\right) \nonumber \\
  & \leq \frac{\pi^{3/2}}{2h(nH_h(X))^{1/2}} e^{-\frac{\sigma^4h^2nH_h(X)}{4\pi^2\gamma^2}},
\end{align}
where the final inequality follows by the Chernoff's bound.
This allows us to deduce, once again using the inequality $e^{-x}<x^{-1/2}$, that this integral is bounded above as
$$
\int_\frac{\sigma^2}{4\gamma}^\frac{\pi}{h} |\varphi(t)|^n dt \leq \frac{\pi^{5/2} \gamma}{h^2\sigma^2 n H_h(X)}.
$$
This allows us to state the following corollary to Theorem~\ref{thm:lllt}.
\begin{corollary}\label{lllt_coro}
 Let $X_1, X_2, \ldots, X_n$ be independent and identically distributed random variables taking values on $\cL(v_0, h)$.
 Define $S_n = \sum_{i=1}^n X_i$.
 Suppose that $\mu = \bE(X_1)=0$, $\sigma^2 = \Vv{X_1}$ and $\gamma = \bE|X_1|^3$.
 Then, 
\begin{align}\label{eq:cor_LLT}
 \sup_{ w \in\cL(nv_0, h)} \left| \bP(S_n = w) - \frac{h}{\sqrt{2\pi n \sigma^2}} \exp \left(-\frac{w^2}{2n\sigma^2}\right) \right| \leq \frac{32h\gamma}{\sigma^4 n} + \frac{6\gamma}{h\sigma^2 n H_h(X_1)}.
\end{align}
\end{corollary}

To give an explicit upper bound on the error probability, we need to deduce that $H_h(X_1)$ is bounded from below.  For $\mathbf{x} = (x_1, x_2, \ldots, x_k) \in\bZ^k$, define
 $$
 w(\mathbf{x}):= \max_{i\neq \ell \neq j} \frac{|x_i - x_\ell|}{\gcd(|x_i - x_\ell|, |x_j-x_\ell|)}. 
 $$
Then the fact that $H_h(X_1)$ is bounded from below is implied by the following lemma,
\begin{lemma}
 Let $X$ be an integer valued random variable supported on a lattice of step $h$ and with atoms $x_1, \ldots, x_k$ not all contained in a non-trivial arithmetic progression of the lattice. Then there exists an absolute constant $C>0$ such that
 $$
 H_h(X) \geq C\cdot \frac{\min_{i\in [k]} \bP(X=x_i)}{k (w(\mathbf{x})h)^2}
 $$
\end{lemma}
\begin{proof}
For $d\in \mathbb{R}$, consider
 \begin{align}\label{eq:D}
D(X,d) & := \inf_{\alpha\in\bR} \bE\langle (X-\alpha)d\rangle^2.
 \end{align}
By \cite[Lemma 1]{Mukhin1}, we have that $D(X,d) \leq H(X,d) \leq 4D(X,d)$. Therefore,
$$
H_h(X)\geq \min_{1/4h\leq d\leq 1/2h} D(X,d)
$$
For all $\mathbf{x}\in \mathbb{Z}^k$ and $\beta,d\in \mathbb{R}$, define
 \begin{align}\label{eq:D_prime}
  S(\mathbf{x},\beta,d) :=  \sum_{i=1}^k \langle \beta + x_i d \rangle
 \end{align}
 Cauchy-Schwartz's inequality implies that
\begin{align}
D(X,d) & := \inf_{\alpha\in\bR} \bE\langle (X-\alpha)d\rangle^2 \nonumber\\
& = \inf_{\beta \in\bR} \sum_{i=1}^k \langle \beta+x_i d \rangle^2 \bP(X=x_i) \nonumber\\
&\geq \min_{i\in[k]} \bP(X=x_i) \inf_{\beta \in\bR} \sum_{i=1}^k \langle \beta+x_i d \rangle^2 \nonumber\\ 
& \geq \frac{\min_{i\in[k]} \bP(X=x_i)}{k} \inf_{\beta \in\bR} S(\mathbf{x},\beta, d)^2.\label{eq:D_from_D_prime}
\end{align}
It thus suffices to bound the infimum of $S$ when $\mathbf{x}$ and $d$ are fixed. The derivative of $S$ with respect to $\beta$ satisfies the following properties:
\begin{itemize}
 \item[(i)] it is well defined for all $\beta$ such that $\langle \beta + x_i d \rangle \not\in \{0,1/2\}$ for all $i\in [k]$;
 \item[(ii)] it is constant between any two consecutive values at which the derivative is undefined;
 \item[(iii)] it takes integer values in $\{-k,\dots,k\}$ anywhere where it is defined.
\end{itemize}
Thus, the minimum of $S$ is attained at $\beta_0$, for which the derivative is not defined. By relabelling the $x_i$, we may assume that $\langle \beta_0 + x_k d \rangle\in \{0,1/2\}$. If $\langle \beta_0 + x_k d \rangle = 1/2$, plugging it in~\eqref{eq:D_from_D_prime} we would get the desired bound and we are done. So we may assume that $\langle \beta_0 + x_k d \rangle = 0$, and, in fact, we can choose $\beta_0= -x_k d$. For $i\in [k-1]$, define $y_i:= x_i-x_k$.
As the $x_i$ are not all contained in a non-trivial arithmetic progression then $\hcf(y_1,y_2,\ldots,y_{k-1}) = h$.
By a simple extension of B\'ezout's Lemma there exist $\lambda_i\in \mathbb{Z}$ with $|\lambda_i| \leq w(\mathbf{x}) h$ for all $i\in [k-1]$ and $\lambda_1 y_1 + \lambda_2 y_2 + \ldots +\lambda_{k-1} y_{k-1}=h$.
Now, using the identities $\langle t \beta \rangle \leq |t| \langle \beta \rangle$ for any $t \in \bZ$ and $\langle \beta_1 + \beta_2 \rangle \leq \langle \beta_1 \rangle + \langle \beta_2 \rangle$, we obtain
\begin{align}\label{eq:D_lower_bound}
 \inf_{\beta\in \mathbb{R}}S(\mathbf{x},\beta,d) = \sum_{i=1}^k \langle \beta_0 + x_i d\rangle = \sum_{i=1}^{k-1} \langle y_i d\rangle \geq \sum_{i=1}^{k-1} \frac{\langle \lambda_i y_i d\rangle}{|\lambda_i|} 
 \geq  \dfrac{\left\langle \sum\limits_{i=1}^{k-1}\lambda_i y_i d\right\rangle}{w(\mathbf{x})h} = \frac{\langle hd\rangle}{w(\mathbf{x})h}
\end{align}
Observing that $\langle hd\rangle=hd$ for all $d\leq 1/2h$, we obtain,
\begin{align*}
H_h(X)\geq \min_{1/4h\leq d\leq 1/2h} D(X,d) \geq \min_{1/4h\leq d\leq 1/2h} \frac{\min_{i\in[k]} \bP(X=x_i)}{k} \left(\frac{\langle hd\rangle}{w(\mathbf{x})h}\right)^2 =  \frac{\min_{i\in[k]} \bP(X=x_i)}{16 k(w(\mathbf{x})h)^2}.
\end{align*}

\end{proof}

\section{Barely subcritical regime for the configuration model}

\subsection{Exploration process}\label{sec:explo_CM}

In this section we introduce a process that given a vertex $v\in [n]$ explores $\CMn$ starting by the component containing $v$. We set a total order of the half edges as follows. For every vertex $v$, consider an arbitrary order of its $d_v$ half-edges. Then, the half edges are ordered, first by its corresponding vertex (using the total order on $[n]$) and then by the order given within the half-edges incident to a vertex.

We will denote by $\cF_t$ the history of the process at time $t$. With a slight abuse of notation, we will assume that $\cF_t$ is the subgraph formed by the partial matching at time $t$. Note that the order of the pairings is determined by the knowledge of the matching. The main random variable we would like to track is $X_t=X_t(v)$, defined as the number of unmatched half-edges incident to $V(\cF_t)$ when the process started at $v$. Note that if $X_t=0$, there are no unpaired half-edges and thus $\cF_t$ is a union of components of $\CMn$ containing the component of $v$.

The \emph{exploration process of $\CMn$ starting at $v\in [n]$} is defined as follows:
\begin{itemize}
\item[1)] Let $\cF_0$ be the single-vertex graph on $\{v\}$ and $X_0=d_v$.
\item[2)] While $V(\cF_t)\neq [n]$,
\begin{itemize}
\item[2a)] If $X_t=0$, choose a uniformly unmatched half-edge and let $u$ be the vertex incident to it. Let $\cF_{t+1}$ be constructed from $\cF_t$ by adding $\{u\}$ as an isolated vertex, and let $X_{t+1}=d_u$.
\item[2b)] Otherwise, choose the smallest unmatched half-edge $e$ incident to $V(\cF_t)$ and pair it with a half-edge $f$ chosen uniformly at random from all the unmatched ones. Let $u$ be the vertex incident to $f$. 
\begin{itemize}
\item[i)] If $u\notin V(\cF_t)$, let $\cF_{t+1}$ be constructed from $\cF_t$ by adding vertex $u$ and edge $ef$ and let $X_{t+1}=X_t+d_u-2$.
\item[ii)] Otherwise, let $\cF_{t+1}$ constructed from $\cF_t$ by adding edge $ef$ and let  $X_{t+1}=X_t-2$.
\end{itemize}
\end{itemize}
\end{itemize}

Note that $X_t$ is measurable with respect to $\cF_t$. We define the following parameters:
\begin{align}
\begin{split}
\eta_{t+1}&\coloneqq X_{t+1}-X_t\;,\\
M_t &\coloneqq X_t+\sum_{u\notin V(\cF_t)} d_u\;,\\
Q_t & = \frac{1}{M_t-1}\sum_{u\notin V(\cF_t)} d_u(d_u-2)\;,\\
R_t & = \frac{1}{M_t-1}\sum_{u\notin V(\cF_t)} d_u(d_u-2)^2\;.
\end{split} 
\end{align}
It is straightforward to check that if $X_t>0$, then
\begin{align}\label{eq:exp_eta1}
\E{\eta_{t+1}\mid \cF_t}=Q_t\quad \text{ and }\quad \E{(\eta_{t+1})^2\mid \cF_t}=R_t\;,
\end{align}
and if $X_t=0$, then
\begin{align}\label{eq:exp_eta2}
\E{\eta_{t+1}\mid \cF_t}=\frac{1}{M_t}\sum_{u\notin V(\cF_t)} 	d_u^2\quad \text{ and }\quad \E{(\eta_{t+1})^2\mid \cF_t}\geq \frac{R_t}{2}\;,
\end{align}
although we will never study the process for $t$ such that $X_t=0$.
\subsection{Stochastic domination and random sums}

Recall the definition of $T=T_n$ given in~\cref{def:T}.

Define the distribution $\beta$ as follows: for every $\ell\in L\coloneqq\{-1,0,1,\dots,n-3\}$,
\begin{align}\label{def:beta}
\p{\beta=\ell} \coloneqq
\begin{cases}
\frac{m}{m-2T}\,\p{\eta=-1} - \frac{2T}{m-2T}&\text{if }\ell=-1\;,\\
\frac{m}{m-2T}\,\p{\eta=\ell}&\text{if }\ell\geq 0\;.
\end{cases}
\end{align}
Let $\varphi_{\beta}(\theta)$, $\theta_0^\beta$, $Q_{\beta}$, $R_{\beta}$ and $T_\beta$ be defined as in~\cref{def:Q,def:R,def:MGF,def:T} replacing $\eta$ by $\beta$. By the choice of $\beta$ all main parameters are asymptotically equal to the original ones, as the following result demonstrates.
\begin{lemma} \label{lem:T_asympt_equal}
For every $k\geq 0$, we have $\varphi^{(k)}_{\beta}(\theta)=(1+o(\theta_0))\varphi^{(k)}(\theta) + o(\theta_0)$.
Moreover, $T_\beta=(1+o(1))T$.
\end{lemma}
\begin{proof}
The first part of the lemma follows directly from
\begin{align}
\varphi^{(k)}_{\beta}(\theta) = \bE(\beta^k e^{\theta \beta}) &= \frac{m}{m-2T}\bE(\eta^k e^{\theta \eta}) - \frac{2T}{m-2T} (-1)^k e^{-\theta}\label{eq:phikbeta2} \\
&= (1+O(T/m))\varphi^{(k)}(\theta) + O(T/m)\nonumber\\
&= (1+o(\theta_0))\varphi^{(k)}(\theta) + o(\theta_0).\label{eq:phikbeta}
\end{align}
where in the last line we used the hypothesis  $T=o(m\theta_0)$ in \cref{thm:CM1}.

For the second part, we split into two cases. If $\Delta|Q| = o(R)$, we are in the setting of~\cref{cor:CM2}. In such case 
$$
\log(\varphi_\beta(\theta_0^\beta))^{-1}\sim \frac{Q_\beta^2}{2R_\beta}\sim \frac{Q^2}{2R}\sim \log(\varphi(\theta_0))^{-1},
$$
(see~\cref{sec:pfcor} for the first and third equivalences) and the result follows from the first part of the lemma.

Otherwise $R=O(\Delta|Q|)$. As $R$ is bounded away from zero by \cref{assum:CM} and $\Delta_n\leq n^{1/6}$, it follows that $|Q|$ is of order at least $n^{-1/6}$. Again all we need to show is that $\log(\varphi_\beta(\theta_0^\beta))^{-1} = (1+o(1)) \log(\varphi(\theta_0))^{-1}$ and then the rest will follow by the first part of the lemma. We do this by bounding $\varphi(\theta_0) - \varphi_\beta(\theta^\beta_0)$.

%
%
%

Using $e^x\geq 1+x$, we have 
$$
0=\bE[\eta e^{\theta_0\eta}]\geq \bE[\eta(1+\theta_0\eta)]= Q+\theta_0 R
$$
and thus 
\begin{align}\label{eq:bound_theta_0}
0< \theta_0 \leq \frac{|Q|}{R}=o(1).
\end{align}
By construction, $\beta$ stochastically dominates $\eta$ and it follows that $\varphi^{(k)}_\beta(\theta)\geq \varphi^{(k)}(\theta)$ for all $k\geq 0$ and $\theta\geq 0$. In particular, 
\begin{align}\label{eq:bound_theta_0^beta}
0< \theta_0^\beta \leq \theta_0
\end{align} 
Combining~\eqref{eq:phikbeta2} for $k=0$,~\eqref{eq:bound_theta_0} and~\eqref{eq:bound_theta_0^beta},
\begin{align}\label{eq:diff_theta_beta}
\varphi_\beta(\theta_0^\beta) - \varphi(\theta_0^\beta) = \frac{2T}{m-2T}(\varphi(\theta_0^\beta) - e^{-\theta_0^\beta} )\leq \frac{2T\theta_0}{m} (1+o(1)) = O\left( \frac{T|Q|}{mR}\right).
\end{align}
As $\varphi_\beta''$ is an increasing function with 
$\varphi_\beta''(0) = (1+o(1))R$ and $\varphi_{\beta}'(\theta_0^\beta)=0$, the fundamental theorem of calculus implies
\begin{align}\label{eq:FTC1}
(\theta_0-\theta_0^\beta) R\leq (1+o(1)) \int^{\theta_0}_{\theta_0^\beta} \varphi_\beta''(t) dt= (1+o(1))\varphi_\beta'(\theta_0) = O\left(\frac{T}{m}\right).
\end{align}
where the last equality follows from~\eqref{eq:phikbeta}.

We have $|\varphi'(t)|\leq |Q|$ for all $t\in [0,\theta_0]$; indeed, $\varphi'$ is increasing with $\varphi'(0) = Q$ and $\varphi'(\theta_0) = 0$, Similarly as before, using~\eqref{eq:FTC1} we conclude that 
\begin{align}\label{eq:FTC2}
\varphi(\theta_0) - \varphi(\theta_0^\beta) =  \int_{\theta_0^\beta}^{\theta_0} \varphi'(t) dt \leq (\theta_0-\theta_0^\beta) |Q| = O\left(\frac{T |Q|}{m R}\right) .
\end{align}
Combining~\eqref{eq:diff_theta_beta} and~\eqref{eq:FTC2},
\begin{align}\label{eq:phi_triangle}
\varphi_\beta(\theta_0^\beta) - \varphi(\theta_0) =  O\left(\frac{T |Q|}{m R}\right).
\end{align}
Recall that $\varphi'(0)=Q<0$. Using the inequality $e^x\leq 1+2x$ for $x\in[0,1]$, we have 
$$
\varphi'\left(\frac{|Q|}{5\Delta}\right) \leq \bE\left[\eta\left(1+ \frac{2|Q|\eta}{5\Delta}\right)\right] \leq Q + \frac{2|Q|R}{5\Delta} <0
$$
where the final inequality holds because $Q<0$ implies $R < 2\Delta$.
It follows that $\theta_0 \geq \frac{|Q|}{5\Delta}$.

By using the fact that $\varphi'(\theta)<0$ for all $\theta \in [0,\theta_0)$ and Taylor expansion of $\varphi(\theta)$ around $\theta=0$, we obtain
\begin{align}
 \varphi(\theta_0)  \leq \varphi\left(\frac{|Q|}{5\Delta}\right)& \leq 1 - \frac{Q^2}{5\Delta} + \sum_{k\geq 2} \frac{\bE[\eta^k]}{ k!}\cdot \left(\frac{|Q|}{5\Delta}\right)^k \nonumber\\
 & \leq 1 - \frac{Q^2}{5\Delta} + \frac{2Q^2}{25\Delta} \sum_{\ell \geq 0} \left(\frac{|Q|}{5}\right)^\ell\nonumber\\
 &\leq  1 - \frac{Q^2}{10\Delta}. \label{eq:phi_th0_bound}
\end{align}
where in the third inequality we used that $\bE[\eta^k] \leq 2 \Delta^{k-1}$ for all $k\in\mathbb{N}$, since $Q< 0$.


Using the bound~\eqref{eq:phi_th0_bound} in the definition of $T$ gives the simple upper bound $T \leq \frac{10\Delta}{Q^2}\log(n)$. By our bounds on $\Delta$ and $Q$, $T^2|Q|= O(n^{-5/6}\log n)=o(m R)$. Thus, substituting this into~\eqref{eq:phi_triangle}, we have
\begin{align}\label{eq:final_bound_theta_0}
\varphi_\beta(\theta_0^\beta) - \varphi(\theta_0) = o\left(\frac{1}{T}\right).
\end{align}

By definition, $\log(\varphi(\theta_0))^{-1} \geq 1/T$. Therefore, 
\begin{align}\label{eq:bounded_away_from_1}
\varphi(\theta_0)\leq e^{1/T} = 1 - \frac{1+o(1)}{T}
\end{align}
Combining~\eqref{eq:final_bound_theta_0} and~\eqref{eq:bounded_away_from_1},
\begin{align*}
\log(\varphi_\beta(\theta_0^\beta))^{-1}\sim 1-\varphi_\beta(\theta_0^\beta)\sim 1-\varphi(\theta_0) \sim \log(\varphi(\theta_0))^{-1},
\end{align*}
concluding the proof of the lemma.


\end{proof}

Let $(\beta_t)_{t\geq 1}$ be a sequence of iid copies of $\beta$. For $s\in N$, define the stochastic process $W_t=W^s_t$ by $W_0=s$ and for $t\geq 0$
\begin{align}\label{def:W}
W_{t+1}=W_t+\beta_t = s + \sum_{i=1}^t \beta_i\;.
\end{align}
Define the stopping time
\begin{align*}
\tau^s_W \coloneqq \inf\{t:\, W^s_t=0\}\;.
\end{align*}

Let $h$ be the largest possible common difference of a progression upon which the limiting degree sequence $D$ is supported, that is
$$
h:=\max\{j: \exists k\text{ s.t. } \bP(D \in \cL(k,j))=1\}
$$
\begin{lemma}\label{lem:UB_UP}
For every $t\geq T_\beta$ and $s=s(n)$ we have
\begin{align}\label{eq:stop_time_exact}
\p{\tau^s_W=t} \leq  2h\cdot se^{\theta_0^\beta s} \left({\varphi_{\beta}''(\theta_0^\beta)}\right)^{-1/2} \frac{(\varphi_{\beta}(\theta_0^\beta))^t}{t^{3/2}}\;.
\end{align}

Moreover, for every $\epsilon>0$ we have that
\begin{align}\label{eq:stop_time_tail}
\p{\tau^s_W\geq (1+\epsilon)T_\beta} = o\left(\frac{ s e^{\theta_0^\beta s}}{\E{D_n e^{\theta_0^\beta D_n }}}\cdot \frac{T_\beta}{n}\right)\;.
\end{align}
\end{lemma}
\begin{proof}
The dependence on $s$ is implicit in all the notation below. Define the following sequences
\begin{align}\label{def:indices}
\begin{split}
I_t&=\{\mathbf{b}=(b_1,\dots, b_t)\in L^t:\, s+b_1+\dots+b_t=0\}\\
\hat{I}_t&=\{\mathbf{b}=(b_1,\dots, b_t)\in I_t:\, s+b_1+\dots+b_i>0, \forall i\in [t-1]\}
\end{split}
\end{align}
We can write
\begin{align}
\begin{split}
\p{W_t = 0} &= \sum_{\mathbf{b}\in I_t} \prod_{i=1}^t \p{\beta_i=b_i}
\quad \text{and}\quad
\p{\tau^s_W = t} = \sum_{\mathbf{b}\in \hat{I}_t} \prod_{i=1}^t \p{\beta_i=b_i}
\end{split}
\end{align} 
A variant of Spitzer's lemma~\cite[Lemma 9]{nachmias2010critical} implies that
\begin{align}\label{eq:spitzer}
\p{\tau^s_W = t} \leq \frac{s}{t}\p{W_t = 0}\;
\end{align} 
We use exponential tilting to bound the probability that $W_t=0$, as in~\cite{nachmias2010critical,riordan2012phase}. Consider the probability distribution $\beta_\theta$ defined for $\ell\in L$ by
\begin{align}\label{def:beta_tilt}
\p{\beta_\theta =\ell}= \frac{e^{\theta \ell}\p{\beta=\ell}}{\varphi_{\beta}(\theta)}\;.
\end{align}
Let $(\beta_{\theta,t})_{t\geq 1}$ be a sequence of iid copies of $\beta_\theta$.
Define the stochastic process $W_{\theta,t}$ by $W_{\theta,0}=s$ and for $t\geq 0$
\begin{align}\label{def:W_tilt}
W_{\theta, t+1}=s + \sum_{i=1}^t \beta_{\theta,i}\;.
\end{align}
Algebraic manipulations give
\begin{align}\label{eq:exp_tilt}
\p{W_t=0} = (\varphi_{\beta}(\theta))^t e^{\theta s} \p{W_{\theta,t}=0}\;.
\end{align}
By definition of $\theta_0^\beta$, $\E{\beta_{\theta_0^\beta}}=\E{\beta e^{\beta \theta_0^\beta}}=0$.
We may write $W_{\theta_0^\beta,t}=s+S_t$, where $S_t=\sum_{i=1}^t Y_i$  and $(Y_i)_{i\in [t]}$ is a collection of iid copies of $\beta_{\theta_0^\beta}$. In particular, we have
\begin{align}\label{eq:moments_beta_tilt}
\begin{split}
\mu&=\E{Y_1}=\varphi_{\beta}'(\theta_0^\beta)=0\;,\\ 
\sigma^2&=\E{Y_1^2}= \frac{\varphi_{\beta}''(\theta_0^\beta)}{\varphi_{\beta}(\theta_0^\beta)}\;,\\
\gamma&=\E{|Y_1|^3}\leq 2+\E{Y_1^3} =\frac{2\varphi_{\beta}(\theta_0^\beta)+\varphi_{\beta}'''(\theta_0^\beta)}{\varphi_{\beta}(\theta_0^\beta)}\;.
\end{split}
\end{align}
where in the inequality, we have used that $Y_1\geq -1$.

We will apply~\cref{lllt_coro} and show that the error term is negligible with respect to the Gaussian probability. Recall that $h$ is the step of the limiting distribution $D$, which by \cref{assum:CM} is also the step of the distribution of $Y_1$.  Since $h$ and $H_h(Y_1)$ (as defined in~\eqref{eq:H_D}) are constants, the order of the first error term in~\eqref{eq:cor_LLT} is at most the order of the second one, and it suffices to bound the latter. \cref{assum:CM} implies that $\sigma^2\geq \bP(\hat{D}_n\neq 2)>0$ for large $n$, and that $\gamma= O(\Delta^{1/2})$. Therefore, for any $t\geq T_\beta$
$$
\frac{\gamma}{\sigma^2 t} = O\left(\frac{1}{\sqrt{\sigma^2 t}}\cdot\sqrt{\frac{\Delta}{T_\beta}} \right)= o\left(\frac{1}{\sqrt{\sigma^2 t}}\right).
$$
where we used that $\Delta=o(T)$ and $T\sim T_\beta$ by \cref{lem:T_asympt_equal}.


Since $\p{Y_1=-1}>0$, we may choose $v_0=-1$. Thus, for sufficiently large $n$, we conclude that for any $w \in \cL(-t,h)$,
\begin{equation}
 \bP(S_t = w) \leq \frac{2h}{\sqrt{2\pi t\sigma^2}} \label{eq:Sn_final_bound_to_use}
\end{equation}
We can now use~\eqref{eq:Sn_final_bound_to_use} with $w=-s$ to obtain
\begin{align}\label{eq:LLT_applic}
\p{W_{\theta_0,t}=0}=\p{S_t=-s}\leq \frac{2h}{\sqrt{2\pi t}}\left(\frac{\varphi_{\beta}(\theta_0^\beta)}{\varphi_{\beta}''(\theta_0^\beta)}\right)^{1/2}
\end{align}



Let us show that $\varphi_\beta(\theta_0^\beta)$ is close to $1$. On the one hand, we will use the inequality $e^x\leq 1+xe^x$ for all $x\in \R$, with equality if and only if $x=0$. Since $\p{\beta=0}\neq 1$, by the choice of $\theta_0^\beta$
\begin{align}\label{eq:phi_theta0<1}
\varphi_{\beta}(\theta_0^\beta) = \E{e^{\theta_0^\beta \beta}}< 1+\theta_0^\beta\E{\beta e^{\theta_0^\beta \beta}}=1
\end{align}
On the other hand, using $e^x\geq 1+x$ for $x\in \R$, that $\theta_0^\beta$ is bounded on $n$ and $\E{\beta}=o(1)$, we obtain
\begin{align}\label{eq:phi_theta0>1}
\varphi_{\beta}(\theta_0^\beta) \geq 1+\theta_0^\beta \E{\beta}= 1+o(1)\;.
\end{align}
Thus, we use the asymptotic equivalence $\log (\varphi_\beta(\theta_0^\beta))^{-1}\sim 1-\varphi_\beta(\theta_0^\beta)$.

Combining~\cref{eq:spitzer,eq:exp_tilt,eq:LLT_applic,eq:phi_theta0<1}, we obtain
\begin{align*}
\p{\tau^s_W=t} \leq 2h\cdot se^{\theta_0^\beta s} \left(\varphi_{\beta}''(\theta_0^\beta)\right)^{-1/2} \frac{(\varphi_{\beta}(\theta_0^\beta))^t}{t^{3/2}}\;.
\end{align*}
proving the first part of the lemma.

For the second statement of the lemma, it suffices to prove it for small enough $\epsilon$, so we may assume $\epsilon\in (0,1)$. Observe that $\p{\tau^s_W=t}\neq 0$ implies that $s=hk-t$ for some $k\in \mathbb{Z}$. 
Since $v_0=-1$ and $h$ are coprime, there are at most $\lceil T/h \rceil$ values $t\in [T]$ such that $\p{\tau^s_W=t}\neq 0$.

As our bound on $\p{\tau^s_W=t}$ is decreasing on $t$, using \cref{eq:phi_theta0<1} it follows that 
\begin{align*}
\p{\tau^s_W\geq (1+\epsilon)T_\beta} 
&= \sum_{t\geq (1+\epsilon) T_\beta\atop t\in \cL(-1,h)} \p{\tau^s_W=t}
\leq 2 s e^{\theta_0^\beta s} \left(\varphi_{\beta}''(\theta_0^\beta)\right)^{-1/2} \frac{(\varphi_{\beta}(\theta_0^\beta))^{(1+\epsilon)T_\beta}}{T_\beta^{3/2}}  \sum_{\ell \geq 0} (\varphi_{\beta}(\theta_0^\beta))^{\ell}\\
&= 2 s e^{\theta_0^\beta s} \left(\varphi_{\beta}''(\theta_0^\beta)\right)^{-1/2} \frac{(\varphi_{\beta}(\theta_0^\beta))^{(1+\epsilon)T_\beta}}{T_\beta^{3/2}(1-\varphi_\beta(\theta_0^\beta))}  =  o\left(\frac{ s e^{\theta_0^\beta s}}{\E{D_n e^{\theta_0^\beta D_n }}}\cdot \frac{T_\beta}{n}\right)\;,
\end{align*}
where in the last equality we used that $T_\beta\sim \frac{1}{\log \varphi^{-1}_\beta(\theta_0^\beta)}\log \left(T^{-3/2}_\beta (\varphi''_\beta(\theta_0^\beta))^{-1/2} \E{D_n  e^{\theta_0^\beta D_n }}n\right)$.
\end{proof}

\subsection{Proof of \cref{thm:CM1}}

Fix $\epsilon>0$ sufficiently small. Define the stopping time  $\tau_X(v)$ as the number of edges in the component of $v$, denoted by $\cC(v)$. That is,
\begin{align*}
\tau_X(v) \coloneqq \inf\{t:\, X_t(v)=0\}\;.
\end{align*}
Note that for every $t\leq 2T\wedge \tau_X(v)$, the distribution $\beta$ stochastically dominates $\eta_t$. Thus, $X_t(v)$ is stochastically dominated by $W^{d_v}_t$. 

Let $\delta=\epsilon/3$. It follows from~\cref{lem:UB_UP,lem:T_asympt_equal} that
\begin{align}\label{eq:tau_X}
\p{\tau_X(v)\geq (1+2\delta)T} \leq \p{\tau_X\geq (1+\delta)T_\beta}\leq \p{\tau^s_W\geq (1+\delta)T_\beta}=o\left(\frac{d_v e^{\theta_0^\beta d_v}}{\E{D_n e^{\theta_0^\beta D_n }}}\cdot\frac{T}{n}\right)\;.
\end{align}

Let $Z$ be the number of components of order at least $(1+\epsilon)T$. For any $\epsilon>0$, we can write 
\begin{align}\label{eq:bound_Z}
Z= \sum_{\cC} \ind{|\cC|\geq (1+\epsilon)T} =  \sum_{v\in [n]} \frac{\ind{|\cC(v)|\geq (1+\epsilon)T}}{|\cC(v)|} \leq \frac{1}{T} \sum_{v\in [n]} \ind{|\cC(v)|\geq (1+\epsilon)T}\;,
\end{align}
where the first sum is over the connected components of $\CMn$.

Since $\cC(v)$ is a connected subgraph, it has at least $|\cC(v)|-1$ edges. Thus, the probability of $|\cC(v)|\geq k$ is bounded from above by the probability $\tau_X(v)\geq k-1$. Using~\cref{eq:tau_X} we obtain
\begin{align*}
\E{Z} \leq \frac{1}{T} \sum_{v\in [n]} \p{\tau_X(v)\geq  (1+\epsilon)T-1} = o\left( \frac{1}{\E{D_n e^{\theta_0^\beta D_n }} n}\right) \sum_{v\in [n]} d_v e^{\theta_0^\beta d_v}= o(1)\;.
\end{align*}
\cref{thm:CM1} follows by Markov's inequality on $Z$.

\subsection{Proof of \cref{cor:CM2}}\label{sec:pfcor}

Recall that $\Delta |Q|=o(R)$ and that $\varphi(\theta)$ is the moment generating function of $\eta$. Thus, $\varphi(0)=1$, $\varphi'(0)=Q$, $\varphi''(0)=R$ and  $\varphi^{(k)}(0)\leq \Delta^{k-3} R$ for all $k\geq 3$. 
This implies that the radius of convergence of $\varphi$ (and so of any of its derivatives) is at least $2|Q|/R$. So, for any $\theta$ with $|\theta|< 2|Q|/R$, we have
\begin{align*}
\varphi'(\theta)&= \varphi'(0)+ \theta\varphi''(0)+O(\theta^2 \varphi'''(0))=Q+\theta R+o(Q)\;.
\end{align*}
By the choice of $\theta_0$, we have $\varphi'(\theta_0)=0$ and $\theta_0\sim |Q|/R$.

We can also write
\begin{align*}
\varphi(\theta_0)&= \varphi(0)+ \theta_0\varphi'(0)+\frac{\theta_0^2 \varphi''(0)}{2}+o(1)\sim 1-\frac{Q^2}{2R}\;.
\end{align*}
and 
\begin{align*}
\log (\varphi(\theta_0))^{-1}&\sim \frac{Q^2}{2R}\;.
\end{align*}
Similar arguments give that $\varphi''(\theta_0)\sim R$.

Finally, observe that  for any $\delta>0$,
\begin{align}
\E{De^{D\theta_0}}\leq  \E{De^{(1+\delta)\Delta Q/R}}=  \E{De^{o(1)}}=O(1)\;.
\end{align}
Using all previous estimations, we can write
\begin{align*}
T\sim \frac{2R}{Q^2}\log \left(\frac{|Q|^3n}{R^2}\right)\;.
\end{align*}
It is straightforward to check that, in this case, the condition $\theta_0 m\geq \omega(n) T_n$ is equivalent to $Q\leq -\omega(n) n^{-1/3}R^{2/3}$.

Note that the condition $\Delta_n\leq n^{1/6}$ is only required in~\cref{lem:T_asympt_equal} in the case $R=O(\Delta |Q|)$. Thus, the desired result follows from \cref{thm:CM1} without further restrictions on the degree sequence.

\section{Subcritical regime for the uniform model}

\subsection{Exploration process}

We will use the exploration process described in~\cite{jprr2018} that, given $V_0\subset [n]$, reveals the components of $\Gn$ one by one starting with the components containing $V_0$. 

We first describe the exploration process on a fixed graph where each vertex has an order in its adjacency list. Precisely, an \emph{input} is a pair $(G,\Pi)$, with $G$ a graph on $[n]$ and $\Pi=(\pi_v)_{v\in [n]}$ a collection of permutations where $\pi_v$ has length $d_v$ and induces a natural order on the edges incident to $v$. The process constructs a sequence of sets $V_0\subset V_1\subset \dots$ such that at time $t$ all the edges in $G[V_t]$ have been revealed. Similarly as before, we define $X_t=X_t(v)=|E(V_t,[n]\setminus V_t)|$ to be the number of edges between the explored and unexplored parts. If $X_t=0$, $V_t$ is a set of vertices forming a union of components, including the ones intersecting $V_0$. We also define $E(A,B)$ to be the set of edges between sets $A$ and $B$, $M_t=\sum_{w\in [n]\setminus V_t} d_w$, and we let $L_t$ be the number of vertices of degree $1$ in $[n]\setminus V_t$.

The \emph{exploration process of $(G,\Pi)$ starting at $V_0\subset [n]$} is defined as follows:
\begin{itemize}
\item[1)] Let $X_0=|E(V_0,[n]\setminus V_0)|$.
\item[2)] While $V_t\neq [n]$,
\begin{itemize}
\item[2a)] If $X_t=0$, choose a vertex $u$ in $[n]\setminus V_t$ according to the degree distribution and let $w_{t+1}=u$, i.e. $\p{w_{t+1}=u}=\frac{d_u}{M_t}$. Let $V_{t+1}=V_t\cup \{w_{t+1}\}$ and $X_{t+1}=d_{w_{t+1}}$.
\item[2b)] Otherwise, choose $v_{t+1}$ the smallest vertex incident to at least one edge in $[n]\setminus V_t$ and let $e_{t+1}$ be the smallest edge in $E(v_{t+1},[n]\setminus V_t)$. Let $w_{t+1}$ be the endpoint of $e_{t+1}$ in $[n]\setminus V_t$. Expose all edges in $E(w_{t+1},V_t)$. Let $V_{t+1}=V_t\cup \{w_{t+1}\}$ and $X_{t+1}=X_t-1+d_{w_{t+1}}-|E(w_{t+1}, V_t)|$.
\end{itemize}
\end{itemize}
There are two main differences between this exploration process and the one defined in Section~\ref{sec:explo_CM}: we explore vertex by vertex instead of edge by edge, and we start from a set instead of a single vertex.

\newcommand{\pt}[1]{{\mathbb{P}_t}(#1)}
\newcommand{\Et}[1]{{\mathbb{E}_t}\left[#1\right]}
\newcommand{\Ei}[1]{{\mathbb{E}_i}\left[#1\right]}

We will run the exploration process on an input $(G,\Pi)$ chosen uniformly at random from all the inputs where $G$ is a graph on $[n]$ with degree sequence $\bfd_n$. This is equivalent to sampling $G\sim \Gn$ and, independently, letting $\Pi=\Pi(\bfd_n)$ be a collection of uniformly and independent permutations of lengths $(d_v)_{v\in [n]}$. We will use the \emph{principle of deferred decisions} exposing the restriction of $\pi_{v_{t+1}}$ onto $E(v_{t+1},[n]\setminus V_t)$ at time $t$. Let $(\cF_t)_{t\geq 0}$ be the filtration of the space of inputs given by the history of the process just after exposing the order on $E(v_{t+1},[n]\setminus V_t)$. The random objects $X_t$, $V_t$, $M_t$, $L_t$, $v_{t+1}$ and $e_{t+1}$ are $\cF_t$-measurable, while $w_{t+1}$ is $\cF_{t+1}$-measurable. We will use $\pt{\cdot}\coloneqq \p{\cdot \mid \cF_t}$ and $\Et{\cdot}\coloneqq \E{\cdot \mid \cF_t}$ to denote respectively the probability and expected value conditioned to $\cF_t$.

\subsection{Deterministic properties of the process}
First of all, we may assume that $m_0=o(n)$, as otherwise since $|Q_0|\leq 1$, there is nothing to prove. This implies that, $m\leq 3n$ as $\sum_{w\in [n]\setminus S_*}d_w\leq 2n$.

Define 
\begin{align}\label{def:T2}
T \coloneqq 
\frac{m_0}{|Q_0|} \geq \frac{\Delta|Q_0|+R}{Q_0^2}\log{\lambda}
\end{align}
where
\begin{align}\label{eq:inside_log}
\lambda\coloneqq \frac{nQ_0^2}{\Delta|Q_0|+R}\geq \frac{\omega(n)}{|Q_0|} \to \infty\;.
\end{align}
The last condition imposed implies that $T=o(|Q|m)$, we will use this bound repeatedly during the proof.

Let $S$ be the set that certifies $(m_0,Q_0)$-subcriticality. We can always assume that such set is formed by vertices of largest degrees. Note that our condition on $m_0$ implies that $\Delta=o(m_0)$. So, we may also assume that 
\begin{align}\label{eq:assump2}
\sum_{v\in S} d_v\geq \frac{m_0}{2}\;,
\end{align}
as increasing the set $S$ can only decrease  $Q_0$.

Throughout the proof, we will assume that 
\begin{align}\label{assump:max_deg>=2}
\Delta'\coloneqq \max_{w\in [n]\setminus S} d_w\geq 2\;.
\end{align}
Otherwise there are at most $m_0$ vertices of degree at least $2$ and any component has order at most $O(m_0)$ and we are done.

\begin{lemma}\label{lem:det}
Let $v\in [n]$ and set $V_0=S\cup\{v\}$. We have:
\begin{enumerate}
\item\label{it:prop1} $\sum_{u\in V_0} d_u\leq 2|Q_0|T$;
\item\label{it:prop3}  $\Delta' =o\left(\frac{n_1}{\Delta}\right)$.
\end{enumerate}
Moreover, for every $t=O(T)$ we have:
\begin{enumerate}\setcounter{enumi}{2}
\item\label{it:prop4}  $n_1(t)\geq n_1/2$;
\item\label{it:prop5}  $M_t\geq m/3$.
\end{enumerate}
\end{lemma}
\begin{proof}
For \cref{it:prop1}, just observe that $d_v\leq \Delta\leq m_0=|Q_0|T$.

For \cref{it:prop3}, since $\Delta'\geq 2$, $m_0\geq 3m_*$ and by~\cref{eq:neg_set,eq:surplus,eq:assump2}, we have
\begin{align*}
0\geq \sum_{v\in [n]\setminus S_*} d_v(d_v-2) \geq  -n_1 + (\Delta'-2)\sum_{w\in S\setminus S_*} d_w \geq -n_1+  (\Delta'-2)(m_0/2-m_*)\;,
\end{align*}
From here it follows that $\Delta'=O(n_1/m_0)=o(n_1/\Delta)$.

For \cref{it:prop4}, observe that $n_1\geq |Q_0|m$ and $T=o(|Q_0|m)$. So $n_1(t)\geq n_1-1-t\geq n_1/2$. 

For \cref{it:prop5}, by  \cref{eq:neg_set} we have
$$
0\geq Qm\geq -n_1+\sum_{v\in [n]\setminus S_*\atop d_v\geq 3} d_v = m-m_*- 2(n_1+n_2)\;.
$$
Counting only the contribution of vertices of degree $1$ or $2$ to $M_t$, we obtain
\begin{align*}
M_t \geq n_1+2n_2- m_* -2t \geq \frac{m}{2}+o(n)\geq \frac{m}{3}\;.
\end{align*}

\end{proof}
\subsection{Bounding the increments}

\begin{lemma}\label{lem:switch}
For any $v\in [n]$, any $t=O(T)$ and any $w\in [n]\setminus V_t$, we have
\begin{align}\label{eq:switch:UB}
\pt{w_{t+1}=w}\leq (1+o(1))\frac{d_w}{M_t}\;.
\end{align}
Moreover, if $d_w=1$
\begin{align}\label{eq:switch:LB}
\pt{w_{t+1}=w}\geq (1+o(1))\frac{d_w}{M_t}\;.
\end{align}
\end{lemma}
\begin{proof}
The proof uses an edge-switching argument. A \emph{switching} is a local operation that transforms an input into another one. Given an input $(G,\Pi)$ and two oriented edges $(a,b)$ and $(c,d)$ with $ab,cd \in E(G)$ and $ac,bd\notin E(G)$, we obtain the new input by deleting the edges $ab$ and $cd$, and adding the edges $ac$ and $bd$. Note that this operation preserves the degree of each vertex and does not modify the permutations of the adjacency lists. We will restrict to switchings that do not modify the edges within $V_{t}$ in order to switch between inputs in $\cF_{t}$.

Fix $w\in [n]\setminus V_{t}$. If $X_t=0$, then $w$ is chosen with probability $d_w/M_t$, so we may assume that  $X_t>0$. Let $v_{t+1}$, $e_{t+1}$ and $w_{t+1}$  as described in the process. Given $\cF_t$, $v_{t+1} $ and $e_{t+1}$ are fixed, while $w_{t+1}$ is a random vertex. Let $\cA\subseteq \cF_{t}$ be the set of inputs with $w_{t+1}=w$ and $\cB=\cF_{t}\setminus \cA$. We will estimate the number of switchings between $\cA$ and $\cB$ to prove the lemma.

We first proof~\cref{eq:switch:UB}. To switch from $\cB$ to $\cA$, we need to switch the edges $(v_{t+1},w_{t+1})$ and $(w,u)$ for $u\in N(w)$ and there are at most $d_w$ such switchings for each input in $\cB$. To switch from $\cA$ to $\cB$, it suffices to select the edges $(v_{t+1},w)$ and $(x,y)$ with $x\notin N(v_{t+1})\cup V_{t}$ and $y\notin N(w)$. By~\cref{lem:det}, there are at most $\Delta \Delta'+\Delta'\Delta= o(n_1)=o(M_t)$ oriented edges $(x,y)$ with $x\in [n]\setminus V_{t}$ that violate the previous condition. Thus, there are at least $(1+o(1))M_t$ switchings for each input in $\cA$.
It follows that
$$
\pt{w_{t+1}=w} = \frac{|\cA|}{|\cA|+|\cB|} \leq \frac{|\cA|}{|\cB|} \leq (1+o(1))\frac{d_w}{M_{t}}\;.
$$

We now prove \cref{eq:switch:LB}. Suppose that $d_w=1$. To switch from $\cA$ to $\cB$ we must choose the oriented edge $(v_{t+1},w)$ and an oriented edge $(x,y)$ with $x\in [n]\setminus V_{t}$, otherwise we would alter the edges within $V_{t}$. It follows that there are at most $M_{t}$ switchings for each input in $\cA$. To switch from $\cB$ to $\cA$, we must choose the oriented edge $(v_{t+1},w_{t+1})$ and the unique oriented edge $(w,u)$, where $u$ is the only neighbour of $w$. Observe that if either $v_{t+1}w$ or $w_{t+1}u$ is an edge of the graph, the switching is invalid. Instead of giving a lower bound for the number of switchings of a fixed input in $\cB$, we will give a lower bound for the average number of switchings over $\cB$. For each $z\in [n]\setminus (V_{t}\cup \{w\})$, let $\cB_z$ be the set of inputs in $\cB$ with $w_{t+1}=z$. Given an input $(G,\Pi)$ and $x\in [n]\setminus V_{t}$ with $d_x=1$, we say that the input is \emph{$x$-good} if $v_{t+1}x,zy \notin E(G)$, where $y$ is the only neighbour of $x$; otherwise we call the input \emph{$x$-bad}. Since $d_x=1$ and $d_z\leq \Delta'$, by~\cref{lem:det}, there are at most $\Delta+\Delta'\Delta =o(n_1)=o(n_1(t))$ vertices $x$ for which a given input is $x$-bad. We can generate a random input in $\cB_z$, by first choosing one uniformly at random and then permuting the labels of the vertices of degree $1$ in $[n]\setminus (V_t\cup \{z\})$. Thus, the probability that a random input in $\cB_z$ is $w$-bad is $o(1)$. If an input is $w$-good, switching $(v_{t+1},w_{t+1})$ with $(w,u)$ yields an input in $\cA$. 
It follows that
$$
\pt{w_{t+1}=w}=\frac{|\cA|}{|\cA|+|\cB|}=\frac{1}{1+|\cB|/|\cA|} \geq \frac{1}{1+(1+o(1))M_{t}} = (1+o(1)) \frac{1}{M_{t}}
$$

\end{proof}

Define $\eta_t = d_{w_t}-2$. Next result bounds the first and second moments of $\eta_t$.
\begin{lemma}\label{lem:eta_moments}
For any $v\in [n]$ and any $t=O(T)$, we have
\begin{align}\label{eq:eta_moments}
\Et{\eta_{t+1}} \leq \frac{Q_0}{2} \quad \text{and} \quad \Et{(\eta_{t+1})^2}\leq 4R\;.
\end{align}
\begin{proof}
Note that $\sum_{i=1}^{t}d_{w_i}(d_{w_i}-2) \geq -t$. Using~\cref{lem:det} and $t=O(T)=o(|Q_0|m)$,
\begin{align}\label{eq:V_t}
\sum_{w\in [n]\setminus V_{t}} d_w(d_w-2)&= \sum_{w\in [n]\setminus V_{0}} d_w(d_w-2)- \sum_{i=1}^{t}d_{w_i}(d_{w_i}-2) \leq Q_0 m+ t+1\leq \frac{Q_0 m}{2}\leq 0\;.
\end{align}
Applying~\cref{lem:switch} and $M_t\leq m$,
\begin{align*}
\Et{\eta_{t+1}}&= \sum_{w\in [n]\setminus V_{t}} (d_w-2)\pt{w_{t+1}=w}
\leq \frac{1+o(1)}{M_{t}}\sum_{w\in [n]\setminus V_{t}}d_w(d_w-2)\leq Q_0/2\;.
\end{align*}
Similarly, we can bound the second moment. By \cref{lem:det} and~\cref{eq:switch:UB},
\begin{align*}
\Et{(\eta_{t+1})^2}
&= \sum_{w\in [n]\setminus V_{t}} (d_w-2)^2\pt{w_{t+1}=w}
\leq \frac{(1+o(1))}{M_t} \sum_{w\in [n]\setminus V_{t}} (d_w-2)^2 d_w \leq 4R
\end{align*}

%

\end{proof}

\end{lemma}

\subsection{Proof of \cref{thm:UM1}}

Let $\gamma\coloneqq 80$. Define the stopping time 
\begin{align*}
\tau_X= \tau_X(v) = \inf\{t: X_t=0\}\wedge (\gamma T+1)\;,
\end{align*}
where $X_t$ is obtained by starting the process with $V_0= S\cup \{v\}$. We omit the floor and ceiling functions in this section for ease of notation.

Instead of studying $X_t$, we focus on the stochastic process $(Z_t)_{t\geq 0}$ defined by $Z_0=2|Q_0|T$ and for $t\in \N$
\begin{align}\label{def:proc_Z}
Z_{t+1}\coloneqq Z_t+\eta_{t+1}= 2|Q_0|T+ \sum_{i=0}^{t} \eta_{i+1}\;.
\end{align}
Observe that $Z_t$ is $\cF_t$-measurable. For any $t<\tau_X$, we can bound the increments $X_{t+1}-X_t\leq d_{w_{t+1}}-2=\eta_{t+1}$. Therefore, for every $t\leq \tau_X(v)$ we have $X_{t+1}\leq Z_{t+1}$.

Define the stopping time
\begin{align}\label{def:stop_Z}
\tau_Z=\tau_Z(v) \coloneqq \inf\{t: Z_t=0\}\wedge (\gamma T+1)\;,
\end{align}
where $Z_t$ is obtained by starting the process with $V_0=S\cup \{v\}$. Hence, $\tau_X(v) \leq \tau_Z(v)$ and it suffices to bound the latter from above.

%

Write $\mu_{t+1} \coloneqq (\eta_{t+1}-\Et{\eta_{t+1}})\ind{t< \tau_Z}$ and $S_{t+1}\coloneqq \sum_{i=0}^{t} \mu_{i+1}$. For every $t<\tau_Z$, we can write
\begin{align}\label{def:proc_Z2}
Z_{t+1}=  2|Q_0|T + S_{t+1} + \sum_{i=0}^{t} \Ei{\eta_{i+1}}\;.
\end{align}

Since $\Ei{\mu_{i+1}}=0$ for all $i\geq 0$, $S_t$ is a martingale with respect to $\cF_t$ with $S_0=0$. We will use the following Bennett-type concentration inequality for martingales due to Freedman.
\begin{lemma}[\cite{f1975}]\label{lem:freedman}
Let $(S_t)_{t\geq 0}$ be a martingale with respect to a filtration $(\cF_t)_{t\geq 0}$ with $S_0=0$ and increments $\mu_{t+1}=S_{t+1}-S_t$. Suppose there exists $c>0$ such that $\max_{t\geq 0} |\mu_{t+1}|\leq c$ almost surely. For $t\geq 0$, define
$$
V(t+1)\coloneqq \sum_{i=0}^t \Ei{(\mu_{i+1})^2}\;.
$$
Then, for every $\alpha,\beta>0$
$$
\pl{S_t\geq \alpha \text{ and }V(t)\leq \beta\text{ for some }t\geq 1}\leq \exp\left(\frac{-\alpha^2}{2(\beta+c\alpha)}\right)\;.
$$
\end{lemma}

Deterministically, we have $\max_{t\geq 0} |\mu_{t+1}|\leq \Delta\eqqcolon c$. Moreover, by~\cref{lem:eta_moments} for all $t\geq 0$,
\begin{align}\label{eq:bounded_var}
V(t) \leq\sum_{i=0}^{t-1} \Ei{(\eta_{i+1})^2\ind{i<\tau_Z}}\leq 4R (t\wedge \gamma T).
\end{align}
Choose $\alpha= (\gamma/3)|Q_0|T$ and $\beta=4R \gamma T$. Thus, for all $t\geq 0$, $V(t)\leq \beta$ deterministically and, since $|Q_0|\leq 1$, $2(\beta+c\alpha)\leq 8\gamma(R+\Delta |Q_0|)T$. By~\cref{lem:freedman}, uniformly on the choice of $v\in [n]$
\begin{align}\label{eq:bound:S_t}
\p{S_t\geq \alpha \text{ for some } t}& \leq \exp\left(-\frac{\gamma T |Q_0|^2}{72(R+\Delta|Q_0|)}\right)= O\left(1/\lambda\right)\;.
\end{align}
since $\gamma\geq 72$. 

By~\cref{eq:eta_moments} we have $\sum_{i=0}^{\gamma T-1} \Ei{\eta_{i+1}}\leq (\gamma/2)|Q_0|T$. Combining it with \cref{{eq:bound:S_t}}, we obtain uniformly on $v\in [n]$
\begin{align}\label{eq:bound_tau_Z}
\pl{\tau_Z(v)> \gamma T} = \pl{Z_{t}>0 \text{ for all }t\leq  \gamma T} \leq \pl{S_{\gamma T}> (\gamma/2-2) |Q_0|T}= O\left(1/\lambda \right)\;.
\end{align}
since $\gamma\geq 12$.

Observe that if $|\cC(v)|> (\gamma +2)T$, then $\tau_Z(v)\geq \tau_X(v)> \gamma T$. As in~\cref{eq:bound_Z}, letting $Z$ be the number of components of size larger than $(\gamma+2)T$ and by~\cref{eq:bound_tau_Z}
\begin{align*}
\E{Z}\leq \frac{1}{(\gamma+2) T}\sum_{v\in [n]}\p{|\cC(v)|> (\gamma+2) T} \leq \frac{1}{T}\sum_{v\in [n]}\p{\tau_Z(v)> \gamma T} = O\left(1/\log{\lambda}\right)=o(1)\;.
\end{align*}
Markov's inequality concludes the proof.

\section{Proof of \cref{prop:UM2}}\label{sec:exm}


Given $\epsilon\in (0,1)$ and $\Delta=\Delta(n)=o(\sqrt{n})$ with $\log{n}=o(\Delta)$, define 
\begin{align}\label{def:ell_0}
\ell=\left\lfloor (1-\epsilon)\frac{n}{\Delta^2}\right\rfloor.
\end{align}
Consider the degree sequence $\widehat \bfd_n$ that contains $n-\ell$ vertices of degree $1$ and $\ell$ vertices of degree $\Delta$. We may assume that the sum of the degrees is even, otherwise we may add another vertex of degree $1$. For the sake of simplicity, we will omit the floor in the definition of $\ell$. Straightforward computations show that $m= (1+O(\frac{1}{\Delta}))n$ and $Q\sim -\epsilon$.

\newcommand{\pstar}[1]{{\mathbb{P}_*}(#1)}
\newcommand{\ppstar}[1]{{\mathbb{P}_{p_*}}(#1)}
\newcommand{\pp}[1]{{\mathbb{P}_{p}}(#1)}
\newcommand{\Estar}[1]{{\mathbb{E}_*}\left[#1\right]}
\newcommand{\Ep}[1]{{\mathbb{E}_p}\left[#1\right]}
\newcommand{\Epstar}[1]{{\mathbb{E}_{p_*}}\left[#1\right]}

Let $L\subseteq [n]$ denote the set of vertices of degree $\Delta$. Let $\bG_*$ be the random subgraph induced by $\CMn(\widehat \bfd_n)$ on $L$. Let $\bG(L,p)$ be  the Erd\H os-R\'enyi random graph on the vertex set $L$, where each edge in $\binom{L}{2}$ is chosen independently with probability $p$.  Let $\pstar{\cdot}$ and $\pp{\cdot}$  be the probability measures on (multi)graphs with vertex set $L$ associated to $\bG_*$ and $\bG(L,p)$, respectively, and let $\Estar{\cdot}$ and $\Ep{\cdot}$ the expected value operator defined in these probability spaces. 

We briefly sketch the proof. Most of the half-edges in $\widehat \bfd_n$ are incident to vertices of degree $1$. So typically, all vertices in $L$ will pair most of their half-edges with the ones incident to $V\setminus L$ and the order of the largest component in $\CMn(\widehat \bfd_n)$ will be of order at least $\Delta L_1(\bG_*)$.  To estimate $L_1(\bG_*)$, we will show that $\bG_*$ behaves like $\bG(L,p_*)$ with $p_*\coloneqq\frac{\Delta^2}{n}= \frac{(1-\epsilon)}{\ell}$.  Classic results on the subcritical regime of random graphs will give lower bounds for $L_1(\bG(L,p_*))$ that also apply to $L_1(\bG_*)$. We will finally use~\cref{eq:simple} to transfer the lower bound on the largest component from  $\CMn(\widehat \bfd_n)$ to $\Gn(\widehat \bfd_n)$.

Precisely, we will show that certain small subgraphs in $\bG_*$ appear with the same probability as in $\bG(L,p_*)$.
Let $Z_s$ be the number of isolated trees of size $s$ in $\bG_*$. Paley-Zygmund's inequality implies
\begin{align}\label{eq:main_LB}
\pstar{Z_s>0} \geq \frac{\Estar{Z_s}^2}{\Estar{Z_s^2}}\;.
\end{align}

%
%
%
%
%
%
%
%

\begin{lemma}\label{lem:H}
For every $s=O(\log{\ell})$ we have,
\begin{align*}
\Estar{Z_s}&= (1+o(1))\Epstar{Z_s}\\
\Estar{Z_s^2}&= (1+o(1))\Epstar{Z_s^2}\;.
\end{align*}
\end{lemma}
\begin{proof}
Choose $S\subset L$ with $|S|=s$ and any tree $T$ with $V(T)=S$. Let $A_T$ be the event that $S$ induces an isolated copy of $T$ in $L$, which can be defined for $\bG_*$ and $\bG(L,p)$. 

Fix an arbitrary ordering of $E(T)$, $e_1,\dots,e_{s-1}$. A \emph{realisation} of $T$ is a set of pairs of half-edges $\{a_1b_1,\dots,a_{s-1}b_{s-1}\}$ such that the endpoints of $e_i$ are the vertices incident to $a_i$ and $b_i$. Let $k(T)$ be the number of realisations of $T$. 
If $d_1,\dots, d_s$ is the degree sequence of $T$, then $\sum_{i=1}^s d_i= 2(s-1)$ and
\begin{align}\label{eq:k}
k(T)=\prod_{i=1}^s \frac{\Delta!}{(\Delta-d_i)!} =\Delta^{2(s-1)} \prod_{i=1}^s \left(1+O\left(\frac{s}{\Delta}\right)\right)= (1+o(1))\Delta^{2(s-1)}\;,
\end{align}
since $s^2=o(\Delta)$. 

%
%

The event $A_T$ admits a partition into $k(T)$ subevents $A^1_T,\dots,A_T^{k(T)}$  depending on the realisation of $T$.
For $i\in[k (T)]$, $\pstar{A_T^i}$ is equal to the probability that $\CMn$ satisfies:
\begin{enumerate}
\item[(P1)] the $i$-th realisation of $T$ is in $\CMn$;
\item[(P2)] for every $u\in S$ and every incident half-edge $a$ not in the $i$-th realisation, $a$ is paired in $\CMn$ to a half-edge incident to $V\setminus L$.
\end{enumerate}

Let $r=s(\ell-s)+\binom{s}{2}$. For $i\in [k]$, consider the $i$-th realisation of $T$, let $(a_1^ib_1^i,\dots, a_{s-1}^ib_{s-1}^i)$ be a sequence of pairings corresponding to $E(T)$ and let $(\bar{a}_{s}^i \bar{b}_{s}^i,\dots, \bar{a}_{r}^i\bar{b}_{r}^i)$ be a sequence of all nonpairings with at least one half-edge in $S$ (we will assume $\bar{a}_j^i$ is always incident to $S$). 
Let $B_j$ be the event that  $a^i_lb^i_l$ is a pairing for all $l\leq  j\wedge (s-1)$ and $\bar{a}_l^i \bar{b}_l^i $ is not a pairing for all $s\leq l \leq r$. 

We can write
\begin{align}\label{eq:global}
\pstar{A_T}&= \sum_{i=1}^{k(T)} \prod_{j=1}^{s-1} \p{a_j^ib_j^i\in E(\CMn)\mid B_{j-1}} \prod_{j=s}^{r} \p{\bar{a}_{j}^i \bar{b}_{j}^i\notin E(\CMn)\mid B_{j-1}}
\end{align}
We first estimate the probability of (P1).
Each term on the first product in \cref{eq:global} is $\frac{1}{m-O(s)}= \left(1+ O(\frac{s}{m}+\frac{1}{\Delta})\right)\frac{1}{n}$; so the first product is 
\begin{align}\label{eq:first}
\begin{split}
 \prod_{j=1}^{s-1} \p{a^i_jb^i_j\in E(\CMn)\mid B_{j-1}} = \left(1+ O\left(\frac{s^2}{m}+\frac{s}{\Delta}\right)\right)\frac{1}{n^{s-1}}
\end{split}
\end{align}
In order to estimate the probability of (P2) (which is given by the second product in \cref{eq:global}), we compute the probability that each half-edge $a$ incident to $S$ is not paired with half-edges in $L$. There are exactly $s\Delta-2(s-1)$ such events, and each has probability $1-\frac{\ell\Delta-O(s\Delta)}{m+O(s\Delta)}= \left(1-\frac{\ell\Delta}{n}\right)\left(1+O(\frac{s\Delta+\ell}{n})\right)$. Thus, we have
\begin{align}\label{eq:second}
\prod_{j=s}^{r} \p{f_jf_j'\notin E(\CMn)\mid B_{j-1}} 
&=  \left(1-\frac{\ell\Delta}{n}\right)^{s\Delta-2(s-1)}\left(1+O\left(\frac{s\Delta+\ell}{n}\right)\right)^{s\Delta}\nonumber\\
&=  \left(1-\frac{\Delta^2}{n}\right)^{s\ell}e^{O(s/\ell+s/\Delta)} \left(1+O\left(\frac{s^2}{\ell}+\frac{s}{\Delta}\right)\right)\nonumber\\
&= (1+o(1)) \left(1-p_*\right)^{r-s+1}\;,
\end{align}
where we used that $(1-x/N)^y=(1-y/N)^x e^{O((x^2y+y^2x)/N^2)}$ with $x=\ell,y=\Delta, N=n/\Delta$, and that $s=o(\Delta), s^2=o(\ell)$.

Plugging \cref{eq:first,eq:second,eq:k} into \cref{eq:global}, we obtain
\begin{align*}
\pstar{A_T}&= (1+o(1))p_*^{s-1}(1-p_*)^{r-s+1}= (1+o(1))\ppstar{A_T}\;.
\end{align*}
Adding over all sets $S\subset L$ with $|S|=s$ and over all trees $T$ with $V(T)=S$, we obtain the first part of the lemma.

For the second part, choose $S,S'\subset [m]$ with $|S|=|S'|=s$ and any pair of trees $T$ and $T'$ with $V(T)=S$ and  $V(T')=S'$. Note that $\pstar{A_T,A_{T'}}=0$ unless $S=S'$ and $T=T'$, or $S\cap S'=\emptyset$. Suppose we are in the latter case, and let $r=2s(m-2s)+\binom{2s}{2}$. Following similar computations as the ones we did for a single tree, we obtain
\begin{align*}
\pstar{A_T,A_{T'}}&=(1+o(1))\ppstar{A_T,A_{T'}}\;,
\end{align*}
and adding over all pairs of sets and trees supported on these sets, the second part also follows.
\end{proof}

The moments of $Z_s$ in $\bG(L,p)$ are well-studied in random graph theory. Let $I(\lambda)=\lambda-1-\ln \lambda$ be the large deviation rate function for Poisson random variables with mean $\lambda>0$. For $\lambda=1-\epsilon$, any $a<(I_{\lambda})^{-1}$ and $s_0= \lfloor a\log{\ell}\rfloor$, we have $\Epstar{Z_{s_0}^2} = (1+o(1))\Epstar{Z_{s_0}}^2$ (see e.g. Lemma 2.12(i) in~\cite{frieze2016introduction}). 
Combining this with \cref{lem:H} and~\cref{eq:main_LB}, $\bG_*$ has with high probability an isolated tree of size $s_0$. As every vertex in $L$ has degree $\Delta$, there are exactly $\Delta s_0- 2(s_0-1)$ vertices of degree $1$ that attach to the given tree. Therefore, there exists a component in $\CMn(\widehat \bfd_n)$ of order $(1+o(1))\Delta s_0$. 

Observe that $I_\lambda=\frac{\epsilon^2}{2}+ O(\epsilon^3)$ and since $Q\sim -\epsilon$, we have $I_\lambda\sim \frac{Q^2}{2}$. As $\E{D^2}=O(1)$ and $R\sim \Delta$, we can use \cref{thm:simple} to deduce that 
\begin{align*}
L_1(\Gn(\widehat \bfd_n))\geq (1+o(1))\Delta s_0 \geq (1+o(1))\frac{2R}{Q^2} \log{\left(\frac{n}{R^2}\right)}\;,
\end{align*}
with probability $1-o(1)$.
This concludes the proof of the proposition.

\begin{remark}[Concentration of $L_1(\Gn(\bfd_n))$]\label{rem:non_concentrated}
\cref{prop:UM2} imposes the condition $\Delta=o(\sqrt{n})$, or equivalently $\ell\to \infty$ as $n\to \infty$. If $\Delta$ is of order $\sqrt{n}$, it is easy to check that the probability that $\bG_*=H$ is bounded away from $0$ for every $H$ of order $\ell$. Since the size of the largest component is asymptotically equal to $\Delta L_1(\bG_*)$, $L_1(\CMn(\widehat \bfd_n))$ and $L_1(\Gn(\widehat \bfd_n))$ are not concentrated.
\end{remark}
\begin{remark}[The case $Q=o(1)$]\label{rem:Qto0}
The largest component of Erd\H os-R\'enyi is well-studied in the barely subcritical regime (see e.g. Theorem 5.6 in~\cite{JLR}). If $p=\frac{1-\epsilon(\ell)}{\ell}$ with $\epsilon(\ell)>0$ and $\ell^{-1/3}\ll \epsilon(\ell)\ll 1$, then
\begin{align}\label{eq:LB_Gnp}
L_1(\bG(\ell,p)) \sim 2\epsilon^2 \log (\epsilon^3 \ell)\;.
\end{align}

Let $\ell=(1-\epsilon(n))\frac{n}{\Delta^2}$ and define the degree sequence $\widehat \bfd_n$ as before. Again, $Q\sim -\epsilon$ and $R\sim \Delta$. In particular $\Delta |Q|=o(R)$ holds. 

Set $p=\frac{1-\epsilon(n)}{\ell}$. The same argument as in the proof of \cref{prop:UM2} and~\cref{eq:LB_Gnp} gives
\begin{align}\label{eq:final}
L_1(\Gn(\widehat \bfd_n)) \geq \Delta L_1(\bG(\ell,p)) \geq (1+o(1)) \frac{2\Delta }{\epsilon^{2}} \log\left(\frac{\epsilon ^3 n}{\Delta^2}\right)\geq (1+o(1)) \frac{2R }{Q^{2}} \log\left(\frac{|Q|^3 n}{R^2}\right)\;.
\end{align}
The condition $\epsilon(\ell)\gg \ell^{-1/3}$ for the validity of~\cref{eq:LB_Gnp} is equivalent to $Q\leq -\omega(n)n^{-1/3}R^{2/3}$, for some $\omega(n)\to \infty$.

\end{remark}

\bibliographystyle{abbrvnat}
\bibliography{subcrit}

\begin{thebibliography}{25}
\providecommand{\natexlab}[1]{#1}
\providecommand{\url}[1]{\texttt{#1}}
\expandafter\ifx\csname urlstyle\endcsname\relax
  \providecommand{\doi}[1]{doi: #1}\else
  \providecommand{\doi}{doi: \begingroup \urlstyle{rm}\Url}\fi

\bibitem[Bollob{\'a}s and Riordan(2015)]{bollobas2015old}
B.~Bollob{\'a}s and O.~Riordan.
\newblock An old approach to the giant component problem.
\newblock \emph{J.~Combin.\ Theory (Series B)}, 113:\penalty0 236--260, 2015.

\bibitem[Dhara et~al.(2017)Dhara, {\swap{Hofstad}{van der}}, van Leeuwaarden,
  Sen, et~al.]{dhara2017critical}
S.~Dhara, R.~{\swap{Hofstad}{van der}}, J.~S. van Leeuwaarden, S.~Sen, et~al.
\newblock Critical window for the configuration model: finite third moment
  degrees.
\newblock \emph{Electronic Journal of Probability}, 22, 2017.

\bibitem[Dhara et~al.(2020)Dhara, {\swap{Hofstad}{van der}}, van Leeuwaarden,
  Sen, et~al.]{dhara2020heavy}
S.~Dhara, R.~{\swap{Hofstad}{van der}}, J.~S. van Leeuwaarden, S.~Sen, et~al.
\newblock Heavy-tailed configuration models at criticality.
\newblock In \emph{Annales de l'Institut Henri Poincar{\'e}, Probabilit{\'e}s
  et Statistiques}, volume~56, pages 1515--1558. Institut Henri Poincar{\'e},
  2020.

\bibitem[Doney(1997)]{Doney}
R.~A. Doney.
\newblock One-sided local large deviation and renewal theorems in the case of
  infinite mean.
\newblock \emph{Probability theory and related fields}, 107\penalty0
  (4):\penalty0 451--465, 1997.

\bibitem[Durrett(2007)]{durrett2007random}
R.~Durrett.
\newblock \emph{Random graph dynamics}, volume 200.
\newblock Cambridge university press Cambridge, 2007.

\bibitem[Durrett(2019)]{Durrett}
R.~Durrett.
\newblock \emph{Probability: theory and examples}, volume~49.
\newblock Cambridge university press, 2019.

\bibitem[Erd{\H{o}}s and R{\'e}nyi(1960)]{erdHos1960evolution}
P.~Erd{\H{o}}s and A.~R{\'e}nyi.
\newblock On the evolution of random graphs.
\newblock \emph{Publ. Math. Inst. Hung. Acad. Sci}, 5\penalty0 (1):\penalty0
  17--60, 1960.

\bibitem[Freedman(1975)]{f1975}
D.~A. Freedman.
\newblock On tail probabilities for martingales.
\newblock \emph{the Annals of Probability}, pages 100--118, 1975.

\bibitem[Frieze and Karo{\'n}ski(2016)]{frieze2016introduction}
A.~Frieze and M.~Karo{\'n}ski.
\newblock \emph{Introduction to random graphs}.
\newblock Cambridge University Press, 2016.

\bibitem[Hatami and Molloy(2012)]{hatami2012scaling}
H.~Hatami and M.~Molloy.
\newblock The scaling window for a random graph with a given degree sequence.
\newblock \emph{Random Structures Algorithms}, 41:\penalty0 99--123, 2012.

\bibitem[{\swap{Hofstad}{van der}}(2016)]{vanderhofstad2016}
R.~{\swap{Hofstad}{van der}}.
\newblock \emph{Random {{Graphs}} and {{Complex Networks}}}, volume~1 of
  \emph{Cambridge Series in Statistical and Probabilistic Mathematics}.
\newblock {Cambridge University Press}, {Cambridge, England}, 2016.
\newblock \doi{10/ggv8q7}.

\bibitem[{\swap{Hofstad}{van der}} et~al.(2019){\swap{Hofstad}{van der}},
  Janson, and Luczak]{van2019component}
R.~{\swap{Hofstad}{van der}}, S.~Janson, and M.~Luczak.
\newblock Component structure of the configuration model: barely supercritical
  case.
\newblock \emph{Random Structures Algorithms}, 55\penalty0 (1):\penalty0 3--55,
  2019.

\bibitem[Janson(2008)]{janson2008largest}
S.~Janson.
\newblock The largest component in a subcritical random graph with a power law
  degree distribution.
\newblock \emph{The Annals of Applied Probability}, 18\penalty0 (4):\penalty0
  1651--1668, 2008.

\bibitem[Janson(2009)]{janson2009probability}
S.~Janson.
\newblock The probability that a random multigraph is simple.
\newblock \emph{Combinatorics, Probability and Computing}, 18\penalty0
  (1-2):\penalty0 205--225, 2009.
\newblock \doi{10/bg4m2c}.

\bibitem[Janson and Luczak(2009)]{janson2009new}
S.~Janson and M.~J. Luczak.
\newblock A new approach to the giant component problem.
\newblock \emph{Random Structures Algorithms}, 34:\penalty0 197--216, 2009.

\bibitem[Janson et~al.(2011)Janson, \L{}uczak, and Ruci\'nski]{JLR}
S.~Janson, T.~\L{}uczak, and A.~Ruci\'nski.
\newblock \emph{Random graphs}, volume~45.
\newblock John Wiley \& Sons, 2011.

\bibitem[Joos et~al.(2018)Joos, Perarnau, Rautenbach, and Reed]{jprr2018}
F.~Joos, G.~Perarnau, D.~Rautenbach, and B.~Reed.
\newblock How to determine if a random graph with a fixed degree sequence has a
  giant component.
\newblock \emph{Probability Theory and Related Fields}, 170\penalty0
  (1-2):\penalty0 263--310, 2018.

\bibitem[Joseph(2014)]{joseph2010component}
A.~Joseph.
\newblock The component sizes of a critical random graph with pre-described
  degree sequence.
\newblock \emph{The Annals of Applied Probability}, 24:\penalty0 2560--2594,
  2014.

\bibitem[Molloy and Reed(1995)]{molloy1995critical}
M.~Molloy and B.~Reed.
\newblock A critical point for random graphs with a given degree sequence.
\newblock \emph{Random Structures Algorithms}, 6:\penalty0 161--180, 1995.

\bibitem[Mukhin(1985)]{Mukhin2}
A.~Mukhin.
\newblock Local limit theorems for distributions of sums of independent random
  vectors.
\newblock \emph{Theory of Probability \& Its Applications}, 29\penalty0
  (2):\penalty0 369--375, 1985.

\bibitem[Mukhin(1992)]{Mukhin1}
A.~Mukhin.
\newblock Local limit theorems for lattice random variables.
\newblock \emph{Theory of Probability \& Its Applications}, 36\penalty0
  (4):\penalty0 698--713, 1992.

\bibitem[Nachmias and Peres(2010)]{nachmias2010critical}
A.~Nachmias and Y.~Peres.
\newblock Critical percolation on random regular graphs.
\newblock \emph{Random Structures Algorithms}, 36:\penalty0 111--148, 2010.

\bibitem[Petrov(1975)]{Petrov}
V.~V. Petrov.
\newblock \emph{Sums of independent random variables}.
\newblock Springer-Verlag, New York-Heidelberg, 1975.
\newblock Translated from the Russian by A. A. Brown, Ergebnisse der Mathematik
  und ihrer Grenzgebiete, Band 82.

\bibitem[Pittel(2008)]{pittel2008largest}
B.~G. Pittel.
\newblock On the largest component of a random graph with a subpower-law degree
  sequence in a subcritical phase.
\newblock \emph{The Annals of Applied Probability}, pages 1636--1650, 2008.

\bibitem[Riordan(2012)]{riordan2012phase}
O.~Riordan.
\newblock The phase transition in the configuration model.
\newblock \emph{Comb.,\ Probab.\ Comput.}, 21:\penalty0 265--299, 2012.

\end{thebibliography}

\end{document}